\documentclass[11pt,a4paper,reqno]{amsart} 
\usepackage{footnote}
\usepackage{amssymb}
\usepackage{latexsym}
\usepackage{amsmath}
\usepackage{mathrsfs}
\usepackage[X2,T1]{fontenc}
\usepackage[applemac]{inputenc}
\usepackage{cite}

\usepackage{color}

\usepackage{calc}                   

\newcommand{\scal}[2]{\langle #1,#2\rangle}
\newcommand{\rr}[1]{\mathbf R^{#1}}
\newcommand{\zz}[1]{\mathbf Z^{#1}}

\newcommand{\GL}{\operatorname{\mathcal M}}

\newcommand{\nm}[2]{\Vert #1\Vert _{#2}}
\newcommand{\NM}[2]{\left \Vert #1\right \Vert _{#2}}
\newcommand{\nmm}[1]{\Vert #1\Vert }

\newcommand{\op}{\operatorname{Op}}

\newcommand{\sets}[2]{\{ \, #1\, ;\, #2\, \} }
\newcommand{\Sets}[2]{\left \{ \, #1\, ;\, #2\, \right \} }

\newcommand{\ep}{\varepsilon}
\newcommand{\fy}{\varphi}
\newcommand{\cdo}{\, \cdot \, }
\newcommand{\supp}{\operatorname{supp}}

\newcommand{\ON}{\operatorname{ON}}

\newcommand{\tp}{\operatorname{Tp}}
\newcommand{\vrum}{\vspace{0.1cm}}

\newcommand{\Tr}{\operatorname{Tr}}

\newcommand{\nn}[1]{{\mathbf N}^{#1}}

\newcommand{\maclB}{\mathcal B}
\newcommand{\maclH}{\mathcal H}

\newcommand{\maclK}{\mathcal K}
\newcommand{\maclL}{\mathcal L}

\newcommand{\mascE}{\mathscr E}
\newcommand{\mascF}{\mathscr F}
\newcommand{\mascI}{\mathscr I}

\newcommand{\mascS}{\mathscr S}

\numberwithin{equation}{section}          

\newtheorem{thm}{Theorem}
\numberwithin{thm}{section}

\newcommand{\rubrik}{}
\newtheorem{prop}[thm]{Proposition}
\newtheorem{cor}[thm]{Corollary}
\newtheorem{lemma}[thm]{Lemma}

\theoremstyle{definition}

\newtheorem{defn}[thm]{Definition}
\newtheorem{example}[thm]{Example}

\theoremstyle{remark}
\newtheorem{rem}[thm]{Remark}

\author{Wolfram Bauer}

\address{Leibniz Universit{\"a}t Hannover, Institut f{\"u}r Analysis,
Welfengarten 1, 30167 Hannover, Germany}

\email{bauer@math.uni-hannover.de}

\author{Robert Fulsche}

\address{Leibniz Universit{\"a}t Hannover, Institut f{\"u}r Analysis,
Welfengarten 1, 30167 Hannover, Germany}

\email{fulsche@math.uni-hannover.de}

\author{Joachim Toft}

\address{Department of Mathematics,
Linn{\ae}us University, V{\"a}xj{\"o}, Sweden}

\email{joachim.toft@lnu.se}

\title{Convolutions of Orlicz spaces and
Orlicz Schatten classes, with applications to
Toeplitz operators}

\begin{document}

\begin{abstract}
Let $\Phi$ be a Young function.
We study convolution properties for symbol classes $s_{A,\Phi}$,
which consist of all $a$ such that the pseudo-differential
operator $\op _A(a)$ is in the Orlicz
Schatten class $\mascI _\Phi (L^2(\rr d))$.
Especially we prove Young type results for such classes.
We apply the results on Toeplitz operators and prove Orlicz
Schatten properties of such operators.
\end{abstract}



\maketitle

\par

\section{Introduction}

\par

In the paper we deduce continuity properties
for convolutions between Orlicz spaces and
Orlicz type Schatten-von Neumann symbol classes
in pseudo-differential
calculus. As special cases, we obtain well-known
convolution relations between Lebesgue spaces
and classical Schatten-von Neumann classes,
in quantum
harmonic analysis, given in \cite[Proposition 3.2]{Wer}, by R. Werner.
(See also \cite{Toft3} for an alternative approach,
whose style is similar to the present paper.) 

\par

The Schatten-von Neumann theory offers detailed studies of
linear and compact operators acting on a Hilbert space.
For example, consider a linear operator $T$ acting
on a separable Hilbert space $\maclH$. Then $T$ is compact,
if and only if there is an orthonormal basis $\{ f_j\} _{j=1}^\infty$
of eigenvectors to $|T|$, whose corresponding non-negative and
non-increasing eigenvalues
$\lambda (T)=\{ \lambda _j(T)\} _{j=1}^\infty$ tend to zero as $j$ tends
to infinity, i.{\,}e.
\begin{equation}
\label{Eq:CompCondIntro}
\lim _{j\to \infty}\lambda _j(T) = 0.
\end{equation}
(For notations, see \cite{Ho1,Sim} and Section \ref{sec1}.)

\par

The operator $T$ belongs to $\mascI _p(\maclH )$, the
set of Schatten-von Neumann operators on $\maclH$ of order
$p \in (0,\infty )$, if and only if
$\nm T{\mascI _p}\equiv \nm {\lambda (T)}{\ell ^p}$ is finite.
It follows that
$$
\lambda _j(T)\lesssim j^{-\frac 1p}
$$
when $T\in \mascI _p(\maclH )$, which indeed gives more
concrete information of $\lambda _j$ for large $j$, compared
to \eqref{Eq:CompCondIntro}. We also observe that the smaller
$p$, the faster $\lambda _j(T)$ are imposed to approach $0$
at infinity.

\par

This also indicates that by pushing $p$ closer to
zero, the class $\mascI _p(\maclH )$ becomes closer to
the set of finite-rank operators on $\maclH$. 

\par

The presentation of Schatten-von Neumann classes
in \cite{Sim}, by B. Simon, is not limited to 
$\ell ^p$ norm estimates on the eigenvalues $\lambda _j$
above. Instead, any suitable (quasi-)norm on sequences
can be used.

\par

In this paper we focus on the case when $\{ \lambda _j(T)\} _{j=1}^\infty$
belongs to (quasi-Banach) Orlicz spaces. We recall that any Young function
(or even more general, quasi-Young function)
$\Phi$ delivers the Orlicz space $\ell ^\Phi (\mathbf Z_+)$,
and that for $p\in (0,\infty ]$ one has
\begin{equation}
\begin{aligned}
\label{Eq:OrlGivLeb}
\ell ^p(\mathbf Z_+)
&=
\ell ^{\Phi _{[p]}} (\mathbf Z_+)
\quad \text{when}
\\[1ex]
\Phi _{[p]}(t) &=t^p ,\ p<\infty ,
\quad \text{and}\quad
\Phi _{[\infty ]}(t) =
\begin{cases}
0, & t\le 1,
\\[1ex]
\infty , & t>1
\end{cases}
\end{aligned}
\end{equation}
(also in norms). Here $\mathbf Z_+$ is the set of positive integers.

\par

The Schatten-von Neumann class $\mascI _\Phi (\maclH)$
consists of all $T$ above such that
\begin{equation}
\label{Eq:SchattNormIntro}
\nm T{\mascI _\Phi} \equiv \nm {\lambda (T)}{\ell ^\Phi}
\end{equation}
is finite. Here the norm on the right-hand side
is given in Definition \ref{Def:OrliczSpaces1}
in Section \ref{sec1}.
Evidently, \eqref{Eq:OrlGivLeb} gives
$\mascI _p=\mascI _{\Phi _{[p]}}$, so we regain
the usual Schatten-von Neumann classes by choosing
the Young functions in such ways.

\par
 
By straight-forward applications of Schwartz' kernel theorem
and Fourier's inversion formula, we may identify
the spaces $\mascI _\Phi (L^2(\rr d))$ with certain function
and distribution spaces of operator kernels, or with symbols
to pseudo-differential operators.
The Weyl calculus within the theory of
pseudo-differential operators is here of special interest as
it is often suitable in quantizations. For any
$a$ in the Schwartz class $\mascS (\rr {2d})$, the Weyl operator $\op ^w(a)$
is the linear and continuous operator on $\mascS (\rr d)$, given by
$$
(\op ^w(a)f)(x) = (2\pi )^{-d}\iint _{\rr {2d}} a({\textstyle{\frac 12}}(x+y),\xi )
f(y)e^{i\scal {x-y}\xi}\, dyd\xi ,
\quad f\in \mascS (\rr d).
$$
The definition extends to any $a\in \mascS '(\rr {2d})$, in which case
$\op ^w(a)$ is continuous from $\mascS (\rr d)$ to $\mascS '(\rr d)$.
In fact, by Schwartz' kernel theorem and Fourier's inversion formula
it follows that the map $a\mapsto \op ^w(a)$ is a bijection from
$\mascS '(\rr {2d})$ to the set of linear and continuous operators
from $\mascS (\rr d)$ to $\mascS '(\rr d)$.

\par

In particular it follows that $\mascI _\Phi (L^2(\rr d))$
can be identified with the Banach space of distributions
$$
s_\Phi ^w(\rr {2d})
\equiv
\sets {a\in \mascS '(\rr {2d})}{\op ^w(a)\in \mascI _\Phi (L^2(\rr d))},
$$
with norm
\begin{equation}
\label{Eq:WeylSymbOrlNorm}
\nm a{s_\Phi ^w} \equiv (2\pi )^{\frac d2}\nm {\op ^w(a)}{\mascI _\Phi}.
\end{equation}
Then $\mascI _p (L^2(\rr d))$ and its norm, correspond to
$$
s_p^w(\rr {2d})
\equiv
s_{\Phi _{[p]}}^w(\rr {2d})
=
\sets {a\in \mascS '(\rr {2d})}{\op ^w(a)\in \mascI _p(L^2(\rr d))},
$$
and
\begin{equation}
\label{Eq:WeylSymbLebNorm}
\nm a{s_p ^w} \equiv (2\pi )^{\frac d2}\nm {\op ^w(a)}{\mascI _p},
\end{equation}
respectively.
We have included the factor $(2\pi )^{\frac d2}$ on the right-hand sides
of \eqref{Eq:WeylSymbOrlNorm} and
\eqref{Eq:WeylSymbLebNorm}
because then one has $\nm a{s_2^w}=\nm a{L^2}$ (see e.{\,}g. \cite{Toft3}).

\par

The Young type convolution properties established by R. Werner in
\cite[Proposition 3.2]{Wer} are then equivalent to
\begin{equation}
\label{Eq:ConvSchattLeb}
s_p^w*L^q\subseteq s_r^w
\quad \text{and}\quad
s_p^w*s_q^w\subseteq L^r,
\qquad
\frac 1p+\frac 1q=1+\frac 1r,\ p,q,r\in [1,\infty ].
\end{equation}

\par

In Section \ref{sec3} we extend \eqref{Eq:ConvSchattLeb}
and show that if $c,c_0>0$ and the Young functions
$\Phi _0$, $\Phi _1$, and $\Phi _2$ satisfy
\begin{equation}
\label{Eq:SchattOrlCond}
st_1t_2 \le c(\Phi _0^*(s)\Phi _1(t_1) +\Phi _0^*(s)\Phi _2(t_2)
+ \Phi _1(t_1)\Phi _2(t_2)),
\end{equation}
where $s,t_1,t_2\in [0,c_0]$ and  $\Phi_j^*$ denotes the
conjugate Young function, then
\begin{equation}
\label{Eq:ConvSchattOrl}
s_{\Phi _1}^w*L^{\Phi _2}\subseteq s_{\Phi _0}^w
\quad \text{and}\quad
s_{\Phi _1}^w*s_{\Phi _2}^w\subseteq L^{\Phi _0}.
\end{equation}
By choosing $\Phi _1=\Phi _{[p]}$, $\Phi _2=\Phi _{[q]}$
and $\Phi _0=\Phi _{[r]}$, then \eqref{Eq:ConvSchattOrl}
agrees with \eqref{Eq:ConvSchattLeb}.
In this case, \eqref{Eq:SchattOrlCond} holds true under the
conditions on $p$, $q$ and $r$ in \eqref{Eq:ConvSchattLeb}.

\par

More generally, we show that the mapping
properties of the convolutions
in \eqref{Eq:ConvSchattOrl} still hold true
with other pseudo-differential calculi in place
of the Weyl quantization, in the definition of
$s_\Phi ^w$-classes.

\par

In Section \ref{sec4} we deduce dilated convolution
results for Orlicz Schatten-von Neumann classes.
In particular it follows from Theorem \ref{Thm:OrlSchConvL2}
that if $\Phi _0,\Phi _1,\Phi _2$ satisfy \eqref{Eq:SchattOrlCond},
then
$$
a(s\cdo )*b(t\cdo )\in s_{\Phi _0}^w,
\quad \text{when}\quad
a\in s_{\Phi _1}^w,\ b\in s_{\Phi _2}^w,
\ \pm s^{-2}\pm t^{-2}=1,
$$
for some choices of $+$ and $-$ at each $\pm$. More generally,
in Section \ref{sec4} we present a multi-linear version of this
dilated convolution property. An essential ingredient in the
proof of this extension is induction. In order to facilitate the reader
to reach the applications, also taking into account that
the proof is rather comprehensive, we have pushed it to Appendix
\ref{App:A}.

\par

In Section \ref{sec5} we present some applications concerning
Orlicz Schatten-von Neumann properties for Toeplitz operators.
Especially we find that a Toeplitz operator belongs to
$\mascI _\Phi$ if its symbol belongs to $L^\Phi$, and that similar
conclusions can be performed if a suitable dilation of the symbol
belongs to $s_\Phi ^w$.

\par

Finally we remark that in \cite{BaFuTo2} we perform another
approach, based on interpolation techniques, to reach convolution properties
for Orlicz spaces and Orlicz Schatten-von Neumann symbol classes. The final results are similar, but contain differences between the different approaches. Especially, in some aspects, the results here are more general, while in other aspects, the analogous results in \cite{BaFuTo2} are more general.

\par

\section{Preliminaries}\label{sec1}

\par

In this section we recall some facts about Orlicz spaces,
pseudo-differential operators, Toeplitz operators and
Schatten-von Neumann classes of Orlicz types. In the
last part we deduce some invariance properties of
symbol classes of Orlicz Schatten-von Neumann types.

\par

\subsection{Orlicz spaces}

\par

Since our investigations involve quasi-Banach spaces which
fail to be Banach spaces, we first recall the definition of such spaces.

\par

Let $\maclB$ be a vector space.
A \emph{quasi-norm} (of order $r\in (0,1]$) on $\maclB$ is a 
map $f\mapsto \nm f{\maclB}$ from $\maclB$ to $\mathbf R$ such that
\begin{alignat}{2}
\nm f{\maclB}&\ge 0, &
\qquad f&\in \maclB ,
\notag
\intertext{with equality, if and only if $f=0$,}
\nm {\alpha f}{\maclB} &= |\alpha |\cdot \nm f{\maclB}, &
\qquad
f &\in \maclB ,\ \alpha \in \mathbf C,
\label{Eq:QBanDilNorm}
\intertext{and}
\nm {f+g}{\maclB} &\le  2^{\frac 1r-1}\big (\nm f{\maclB}+\nm g{\maclB}\big ), &
\qquad
f,g&\in \maclB .
\label{Eq:QBanTriaIneq}
\intertext{If the topology on $\maclB$ is defined through this quasi-norm,
then $\maclB$ is called a quasi-normed space (of order $r$),
or an $r$-normed space.
\newline
\indent
A complete $r$-normed space is called a \emph{quasi-Banach space} (of order $r$),
or an $r$-Banach space. By Aoki and Rolewi{\'c} in \cite{Aoki,Rol} it follows that
if $\maclB$ above is an $r$-Banach space, then there is
an equivalent quasi-norm to the previous one which additionally satisfies}
\nm {f+g}{\maclB}^r &\le  \nm f{\maclB}^r+\nm g{\maclB}^r, &
\qquad
f,g&\in \maclB .
\tag*{(\ref{Eq:QBanTriaIneq})$'$}
\label{Eq:QBanTriaIneq2}
\end{alignat}
From now on we always assume that the quasi-norm in
\eqref{Eq:QBanDilNorm} and \eqref{Eq:QBanTriaIneq}
is chosen such that \ref{Eq:QBanTriaIneq2} holds. For a general discussion
of quasi-Banach spaces, we refer to \cite{Rolewicz}.

\par

\begin{example}
If $r\in (0,1]$, then $L^r(\rr d)$ is an $r$-Banach space. Evidently,
\ref{Eq:QBanTriaIneq2} holds for its $r$-norm
$$
\nm f{L^r}
\equiv
\left ( \int _{\rr d}|f(x)|^r\, dx \right )^{\frac 1r}.
$$
\end{example}

\medspace

For the definition of Orlicz spaces, we need to recall the
definition of Young functions. We recall that a function $\Phi:[0,\infty ] \to
[0,\infty ]$ is called \emph{convex} if
\begin{equation*}
\Phi(s_1 t_1+ s_2 t_2)
\leq s_1 \Phi(t_1)+s_2\Phi(t_2),
\end{equation*}
when
$s_j,t_j\in \mathbf{R}$
satisfy $s_j,t_j \ge 0$ and
$s_1 + s_2 = 1,\ j=1,2$.

\par

\begin{defn}
\label{Def:YoungFunc}
Let $r\in (0,1]$.
\begin{enumerate}
\item A function $\Phi$ from $[0,\infty ]$ to
$[0,\infty ]$
is called a \emph{Young function} if
the following is true:
\begin{itemize}
\item $\Phi$ is convex;

\vrum

\item $\Phi (0)=0$;

\vrum

\item $\lim
\limits _{t\to\infty} \Phi (t)=\Phi (\infty )=\infty$.
\end{itemize}

\vrum

\item A function $\Phi$ from $[0,\infty ]$ to
$[0,\infty ]$
is called a \emph{quasi-Young function} (of order $r$),
or an \emph{$r$-Young function}, if $\Phi (t)=\Phi _0(t^r)$,
for some Young function $\Phi _0$.
\end{enumerate}
\end{defn}

\par

We observe that $\Phi$ 
in Definition \ref{Def:YoungFunc} might not
be continuous, because we permit
$\infty$ as function value. For example,
for any $a>0$,
$$
\Phi (t)=
\begin{cases}
0,&\text{when}\ t \leq a
\\[1ex]
\infty ,&\text{when}\ t>a
\end{cases}
$$
is convex but discontinuous at $t=a$.

\par

It is clear that $\Phi$ in 
Definition \ref{Def:YoungFunc} is
non-decreasing, because if $0\leq t_1\leq t_2$
and $s\in [0,1]$ is chosen such
that $t_1=st_2$ and $\Phi$ is the same as in
Definition \ref{Def:YoungFunc} (1), then
\begin{equation*}
    \Phi (t_1)=\Phi (st_2+(1-s)0)
    \leq s\Phi (t_2)+(1-s)\Phi (0)
    \leq \Phi (t_2),
\end{equation*}
since $\Phi (0) =0$ and $s\in [0,1]$. Hence every
quasi-Young function is increasing.

\par

\begin{defn}\label{Def:OrliczSpaces1}
Let $\Phi$ be a quasi-Young function and $\mu$
be a positive measure on a measurable set
$\Omega$.
Then the Orlicz space
$L^{\Phi}(\Omega)$ consists
of all (equivalence classes of) measurable functions
$f:\Omega \to \mathbf C$ such that
$$
\nm f{L^{\Phi}}
\equiv
\nm f{L^{\Phi}(\mu)}
\equiv
\inf  \Sets{\lambda>0}
{\int_\Omega \Phi 
\left (
\frac{|f(x)|}{\lambda}
\right )
\, d\mu(x)\leq 1}
$$
is finite. Here $f$ and $g$ in $L^{\Phi}(\mu )$
are equivalent if $f=g$ a.e.
\end{defn}

\par

Evidently, if $\Phi$ and $\Phi _0$ are the same as in
Definition \ref{Def:YoungFunc} (2), then
$$
\nm f{L^\Phi (d\mu)}=\nm {|f|}{L^\Phi (d\mu)}
=(\nm {|f|^r}{L^{\Phi _0} (d\mu)})^{\frac 1r}.
$$

\par

We set $L^\Phi (\rr d)=L^\Phi (d\mu )$ when $\Omega =\rr d$
and $d\mu$ is the Lebesgue measure on $\rr d$.
For various kinds
of spaces of sequences we have the following definition.

\par

\begin{defn}
Let $d\mu$ be the standard discrete measure on the discrete space
$\Lambda$.
\begin{itemize}
\item  $\ell _0'(\Lambda )$ is the set of all (formal) sequences $a=\{a_j\} _{j\in \Lambda}$.

\vrum

\item $\ell _0(\Lambda )$ is the set of $\{a_j\} _{j\in \Lambda}$ such that $a_j\neq 0$ for at most
finite numbers of $j\in \Lambda$.

\vrum

\item $\ell ^\Phi (\Lambda )=L^\Phi (d\mu)$ and $\ell ^p (\Lambda )=L^p (d\mu)$, also in
norms.

\vrum

\item $\ell ^\sharp (\Lambda )$ is the set of all sequences
$\{a_j\} _{j\in \Lambda}$ such that for every $\ep >0$ one has $|a_j|>\ep$
for at most finite numbers of $j\in \Lambda$. The topology of
$\ell ^\sharp (\Lambda )$ is given by the norm $\nm \cdo{\ell ^\infty}$. 
\end{itemize}
\end{defn}

\par

Evidently,
$L^{\Phi _{[p]}}(\rr d)$, $\ell ^{\Phi _{[p]}}(\Lambda )$
and their quasi-norms agree with $L^p(\rr d)$
and $\ell ^p(\Lambda )$ and their quasi-norms, respectively,
when $\Phi _{[p]}$ is defined as in \eqref{Eq:OrlGivLeb}.

\par

We observe that $\ell ^\sharp (\mathbf Z_+)$ consists of
all $\{ \lambda _j\} _{j=1}^\infty \in \ell ^\infty (\mathbf Z_+)$
such that
$$
\lim _{j\to \infty} \lambda _j=0.
$$
(JT: Moved this last part into here from p. 9.)

\par

It is well-known that if $\Phi$ and $\Phi _0$ are the same as in
Definition \ref{Def:YoungFunc} (2), then
$L^{\Phi _0}(d\mu )$ is a Banach space and 
$L^{\Phi}(d\mu )$ is a quasi-Banach space of order $r$
(see e.{\,}g.~ Theorem 3 of III.3.2 and Theorem 10 of III.3.3
in \cite{RaoRen1}).

\par

In most situations we assume that the quasi-Young
functions satisfy the $\Delta _2$-condition
(near the origin), whose definition is recalled
as follows.

\par

\begin{defn}\label{Def:Delta2Cond}
Let $\Phi: [0,\infty ] \to [0,\infty]$ be a quasi-Young function.
\begin{itemize}
\item $\Phi$ is said to satisfy a \emph{$\Delta_2$-condition}
if there exists a constant $C>0$ such that
\begin{equation}\label{Eq:Delta2Cond}
\Phi(2t) \leq C \Phi(t) 
\end{equation}
for every $t\in [0,\infty ]$.

\vrum

\item $\Phi$
is called \emph{positive} if $\Phi (t)>0$ when $t>0$.

\vrum

\item $\Phi$
is said to satisfy a \emph{weak local $\Delta_2$-condition}
or a \emph{weak $\Delta_2$-condition near the origin}, if 
there are
constants $T>0$ and $C>0$ such that \eqref{Eq:Delta2Cond}
holds when $t\in [0,T]$.

\vrum

\item $\Phi$
is said to satisfy a \emph{local $\Delta_2$-condition}
or a \emph{$\Delta_2$-condition near the origin}, if $\Phi$ is
positive and satisfies a weak local $\Delta _2$-condition.
%
%
\end{itemize}
\end{defn}

\par

Since any quasi-Young function $\Phi$ is non-decreasing and has the limit
$\infty$ at infinity, it follows that $\Phi$ must be positive when it satisfies
a $\Delta _2$-condition.

\par

\begin{rem}\label{Rem:Delta2Cond}
Suppose that $\Phi: [0,\infty ] \to [0,\infty]$
is a quasi-Young function which satisfies a local
$\Delta _2$-condition.
Then it follows by straight-forward arguments that
there is a quasi-Young function $\widetilde \Phi$
which satisfies a $\Delta _2$-condition
(on the whole $[0,\infty )$), and such that
$\widetilde \Phi (t)=\Phi (t)$ when $t\in [0,T]$. If in addition
$\Phi$ is a Young function, then $\widetilde \Phi$ can
be chosen as a Young function.
\end{rem}

\par

Several duality properties for Orlicz spaces
can be described in terms of Orlicz spaces
with respect to Young conjugates, given in
the following definition.

\par

\begin{defn}\label{Def:ConjYoungFunc}
Let $\Phi$ be a Young function. Then
the conjugate Young function
$\Phi ^*$ is given
by
\begin{equation}\label{eq-YoungIneq-conjugate}
\Phi ^*(t)
\equiv
\begin{cases}
{\displaystyle{\sup _{s\ge 0} (st - \Phi(s)),}} &
\text{when}\ t \in [0,\infty ),
\\[2ex]
\infty , &
\text{when}\ t=\infty .
\end{cases}
\end{equation}
\end{defn}

\par

\begin{rem}\label{Rem:PhiLeb}
We observe that the conjugate of a Young function is a Young function.
Let $p\in [1,\infty ]$. Its conjugate exponent is denoted by $p'\in [1,\infty ]$,
and should fulfill $\frac 1p+\frac 1{p'}=1$. If $\Phi _{[p]}$ is given by
\eqref{Eq:OrlGivLeb}, then $\Phi _{[p]}^*=\Phi _{[p']}$. In particular,
if $p=1$, then $\Phi _{[1]}(t)=t$ satisfies a $\Delta _2$-condition,
while for $\Phi _{[1]}^*=\Phi _{[\infty ]}$ fails to fulfill any $\Delta _2$-conditions.
\end{rem}

\par

The conjugate Young functions appear naturally when discussing duality
for Orlicz type spaces.
We omit the proof of the following result since it can be found in e.{\,}g.
\cite{RaoRen1}.

\par

\begin{prop}
\label{Prop:OrlDuality}
Let $d\mu$ be a positive measure
on the measurable space $\Omega$.
\begin{enumerate}
\item If $\Phi$ is a Young function, then the map
$$
(f,g)\mapsto (f,g)_{L^2(d\mu)} \equiv \int f(x)\overline{g(x)}\, d\mu (x)
$$
is a continuous sesqui-linear map from
$L^{\Phi}(d\mu  )\times L^{\Phi ^*}(d\mu  )$ to $\mathbf C$.
Furthermore, if
$$
\nmm f \equiv \sup |(f,g)_{L^2(d\mu )}|,
\qquad
f\in L^\Phi (d\mu),
$$
with supremum taken over all $g\in L^{\Phi ^*}(d\mu)$ such that
$\nm g{L^{\Phi ^*}(d\mu )}\le 1$, then
\begin{equation}
\label{Eq:LuxNormEquiv}
\nm f{L^\Phi (d\mu)} \le \nmm f \le 2\nm f{L^\Phi (d\mu)},
\quad
f\in L^\Phi (\mu)\text .
\end{equation}

\vrum

\item If $\Phi$ is a quasi-Young function which satisfies a $\Delta _2$-condition, then
the simple functions are dense in $L^\Phi (d\mu )$.

\vrum

\item  If $\Phi$ is a Young function which
satisfies a $\Delta _2$-condition, then the dual of $L^{\Phi}(d\mu  )$
can be identified with $L^{\Phi ^*}(d\mu  )$ through the form
$(\cdo ,\cdo )_{L^2(d\mu )}$.
%
%
%
\end{enumerate}
\end{prop}

\par

\begin{rem}
\label{Rem:OrlDuality}
If $d\mu$ is the Lebesgue measure on $\Omega =\rr d$, then
Proposition \ref{Prop:OrlDuality} remains true with the set of simple
functions replaced by $C_0^\infty (\rr d)$.
\end{rem}

\par

In the discrete case, there are some additional properties
explained in the following two propositions. The first one shows that
the discrete Orlicz space $\ell ^\Phi (\Lambda )$ is
completely determined by the behavior of $\Phi$ \emph{near origin}. 
We refer to \cite{ToUsNaOz} for the proof, see also \cite{SchFuh}.
Here $h_1(x) \lesssim h_2(x)$ means that there exists $C>0$
such that $h_1(x) \leq C h_2(x)$ for all arguments $x$ in the domain of $h_1$
and $h_2$. We also write $h_1(x)\asymp h_2(x)$ when
$h_1(x) \lesssim h_2(x)\lesssim h_1(x)$. If $\Omega _1$ and $\Omega _2$
are topological spaces, we also use the notation $\Omega _1\hookrightarrow \Omega _2$
when $\Omega _1\subseteq \Omega _2$ and the inclusion
map $\iota : \Omega _1\to \Omega _2$ is continuous.

\par

\begin{prop}
\label{Prop:DiscreteNorms1}
Suppose $\Lambda$ is a discrete space and
$\Phi _1$ and $\Phi _2$ are quasi-Young
functions such that
\begin{equation}
\label{Eq:LocYoungEst}
\Phi _2(t)\lesssim \Phi _1(t)
\quad \text{when}\quad
t\in [0,T],
\end{equation}
for some constant $T>0$. Then
$\ell ^{\Phi _1}(\Lambda )\hookrightarrow \ell ^{\Phi _2}(\Lambda )$, and
\begin{equation}
\nm f{\ell ^{\Phi _2}} \lesssim \nm f{\ell ^{\Phi _1}},
\qquad
f\in \ell ^{\Phi _1}(\Lambda ).
\end{equation}
\end{prop}

\par

As a consequence of Proposition \ref{Prop:DiscreteNorms1}, it follows that
if $\Lambda$ is a discrete space, $\Phi$
is a quasi-Young function and $\widetilde \Phi$ is the same as in
Remark \ref{Rem:Delta2Cond}, then $\ell ^\Phi (\Lambda )$
is equal to $\ell ^{\widetilde \Phi} (\Lambda )$, with equivalent norms.

\par

\begin{prop}
\label{Prop:DiscrNmLocDelCond1}
Let $\Lambda$ be a discrete space and $\Phi$
be a quasi-Young function.
\begin{enumerate}
\item If $\Phi$ is positive, then
$\ell ^\Phi (\Lambda )\hookrightarrow \ell ^\sharp (\Lambda )$.

\vrum

\item If $\Phi$ is not positive, then
$\ell ^\Phi (\Lambda )=\ell ^\infty (\Lambda )$,
with equivalent (quasi-) norms.
\end{enumerate}
\end{prop}

\par

\begin{proof}
First we prove (1). Therefore, suppose that $\Phi (t)>0$ when $t>0$.
If $f=\{ f_j \} _{j\in \Lambda}$
satisfies $|f_j|\ge c$ for some constant $c>0$ and infinitely many
$j\in \Lambda$, then 
$$
\sum _{j\in \Lambda} \Phi \left ( \frac {|f_j|}\lambda \right ) =\infty ,
$$
for every $\lambda >0$, which implies
$\nm f{\ell ^\Phi}=\infty$. By combining these relationships
with Proposition \ref{Prop:DiscreteNorms1}, one obtains (1).

\par

Next we prove (2).
Therefore suppose $\Phi (t_0)=0$ for some $t_0>0$. Then
$\Phi (t)=0$ when $t\in [0,t_0]$. In particular,
near the origin, it follows that $\Phi$
agrees with $\Phi _{[\infty ]}$ in Remark
\ref{Rem:PhiLeb}. The conclusions in (2)
now follows from Proposition \ref{Prop:DiscreteNorms1}.
\end{proof}

\par


%
%
%
%

\par

\subsection{Pseudo-differential operators}
Let $\maclL (V_1,V_2)$ be the set of all linear and
continuous operators from the topological vector space
$V_1$ into the topological vector space $V_2$.
For each $K\in \mascS '(\rr {2d})$, there is a unique
$T_K\in \maclL(\mascS (\rr d),\mascS '(\rr d))$ such that
\begin{equation}
\label{Eq:KernelMap}
\scal {T_Kf}g = \scal K{g\otimes f},
\qquad
f,g\in \mascS (\rr d).
\end{equation}

\par

On the other hand, by the kernel theorem of Schwartz,
the following holds true (cf. e.{\,}g.~ \cite{Ho1}).

\par

\begin{lemma}
\label{Lem:KernelThm}
The map $K\mapsto T_K$ from $\mascS '(\rr {2d})$
to $\maclL(\mascS (\rr d),\mascS '(\rr d))$ is
an isomorphism.
\end{lemma}

\par

Let $A$ belong to $\GL (d,\mathbf R)$, the set of all $d\times d$-matrices with
entries in $\mathbf R$.
Then the pseudo-differential operator $\op _A(a)$ with symbol
$a\in \mascS (\rr {2d})$
is the linear and continuous operator on $\mascS (\rr d)$, defined
by
\begin{equation}\label{e0.5}
(\op _A(a)f)(x)
=
(2\pi  ) ^{-d}\iint a(x-A(x-y),\xi )f(y)e^{i\scal {x-y}\xi }\,
dyd\xi .
\end{equation}
For general $a\in \mascS '(\rr {2d})$, the
pseudo-differential operator $\op _A(a)$ is defined as the continuous
operator from $\mascS(\rr d)$ to $\mascS '(\rr d)$ with
distribution kernel given by
\begin{equation}\label{atkernel}
K_{a,A}(x,y)=(2\pi )^{-\frac d2}(\mascF _2^{-1}a)(x-A(x-y),x-y).
\end{equation}
Here $\mascF _2F$ is the partial Fourier transform of $F(x,y)\in
\mascS '(\rr {2d})$ with respect to the $y$ variable. This
definition makes sense, since the mappings
\begin{equation}\label{homeoF2tmap}
\mascF _2\quad \text{and}\quad F(x,y)\mapsto F(x-A(x-y),x-y)
\end{equation}
are homeomorphisms on $\mascS '(\rr {2d})$.
In particular, the map $a\mapsto K_{a,A}$ is a homeomorphism on
$\mascS '(\rr {2d})$, and a combination of this fact and
Lemma \ref{Lem:KernelThm} gives the following.

\par

\begin{lemma}
\label{Lem:KernelPseudoDiffThm}
Let $A\in \GL (d,\mathbf R)$.
Then the map $a\mapsto \op _A(a)$
from $\mascS '(\rr {2d})$
to $\maclL(\mascS (\rr d),\mascS '(\rr d))$ is
an isomorphism.
\end{lemma}

\par

If $I_d\in \GL (d,\mathbf R)$ is the identity matrix and
$A=\frac 12I_d$, then $\op _A(a)$ is equal to the Weyl
pseudo-differential operator $\op ^w(a)$ for $a$. If instead $A=0$, then the standard
(Kohn-Nirenberg) representation $\op (a)=a(x,D)$ is obtained.

\par

As a consequence of Lemma \ref{Lem:KernelPseudoDiffThm},
it follows that for each $a_1\in \mascS '(\rr {2d})$ and $A_1,A_2\in
\GL (d,\mathbf R)$, there is a unique $a_2\in \mascS '(\rr {2d})$ such that
$\op _{A_1}(a_1) = \op _{A_2} (a_2)$. The relationship between $a_1$ and $a_2$
is given by
\begin{equation}\label{Eq:CalculiTransform}
\op _{A_1}(a_1) = \op _{A_2}(a_2)
\quad \Leftrightarrow \quad
e^{i\scal {A_1D_\xi }{D_x}}a_1(x,\xi )
=
e^{i\scal {A_2D_\xi }{D_x}}a_2(x,\xi ),
\end{equation}
cf. \cite{Ho1,Toft17}.

\par

Let $A\in \GL (d,\mathbf R)$ and $a\in \mascS '(\rr {2d})$ be
fixed. Then $a$ is called a rank-one element with respect to $A$, if
the corresponding pseudo-differential operator is of rank-one,
i.{\,}e.
\begin{equation}\label{Eq:PseudoRankOne}
\op _A(a)f=(2\pi )^{-\frac d2}(f,f_2)_{L^2}f_1, \qquad f\in \mascS(\rr d),
\end{equation}
for some $f_1,f_2\in \mascS '(\rr d)$ and with the $L^2({\bf R}^d)$
inner product $( \cdot , \cdot )_{L^2}$. (We sometimes include factors of
the form $(2\pi)^{\pm \frac d2}$ in order for receiving other convenient
formulae.) By
straight-forward computations it follows that \eqref{Eq:PseudoRankOne}
is fulfilled, if and only if $a=W_{f_1,f_2}^{A}$, where $W_{f_1,f_2}^{A}$
is the $A$-Wigner distribution, defined by the formula
\begin{equation}\label{Eq:WignerDistDef}
W_{f_1,f_2}^{A}(x,\xi ) \equiv \mascF (f_1(x+A\cdo
)\overline{f_2(x-(I-A)\cdo )} )(\xi ),
\end{equation}
which takes the form
$$
W_{f_1,f_2}^{A}(x,\xi ) =(2\pi )^{-\frac d2} \int
f_1(x+Ay)\overline{f_2(x-(I-A)y) }e^{-i\scal y\xi}\, dy,
$$
when $f_1,f_2\in \mascS (\rr d)$. 
From these relations it follows that
\begin{equation}
\label{Eq:AWigPseudoLink2}
\op _A(W_{f_1,f_2}^A) = (2\pi)^{-\frac d2} f_1\otimes \overline {f_2},
\qquad
f_1,f_2\in \mascS '(\rr d),
\end{equation}
when identifying operators with their kernels.

\par

By combining these facts
with \eqref{Eq:CalculiTransform}, it follows that
\begin{equation}\label{Eq:WignerDistRelDiffMat}
e^{i\scal {A_1D_{\xi}}{D_x }} W_{f_1,f_2}^{A_1}
=
e^{i\scal {A_2D_{\xi}}{D_x }}W_{f_1,f_2}^{A_2}
\end{equation}
for each $f_1,f_2\in \mascS '(\rr d)$ and real $d\times d$-matrices
$A_1$ and $A_2$.

\par

For future references we note the link
\begin{multline}\label{Eq:AWigPseudoLink}
(\op _A(a)f,g)_{L^2(\rr d)}=(2\pi )^{-\frac d2}(a,W_{g,f}^A)_{L^2(\rr {2d})},
\\[1ex]
a\in \mascS '(\rr {2d}) \quad\text{and}\quad f,g\in \mascS(\rr d)
\end{multline}
between pseudo-differential operators and Wigner distributions,
which follows by straight-forward computations (see e.{\,}g.
(1.11) in \cite{Toft17}).

\par

Since
the Weyl case is particularly important, we set
$W_{f_1,f_2}^{A}=W_{f_1,f_2}$ when $A=\frac 12\cdot I_d$, i.{\,}e.
$W_{f_1,f_2}$ is the usual (cross-)Wigner distribution of $f_1$ and
$f_2$.
For the Weyl case we also observe the
convenient formula
\begin{equation}
\label{Eq:WeylOpsAdj}
\op ^w(a)^* = \op ^w(\overline a),
\end{equation}
concerning the adjoint of Weyl quantizations
with respect to the $L^2$ form. That is,
$$
(\op ^w(a)f,g)_{L^2(\rr d)}
=
(f, \op ^w(\overline a)g)_{L^2(\rr d)},
\quad
f,g\in \mascS (\rr d).
$$

\par

\subsection{Toeplitz operators}

\par

Next we recall some facts about Toeplitz operators.
The Toeplitz operator $\tp _{\phi _1,\phi _2}(a)$,
with symbol  $a\in \mascS (\rr {2d})$, and window functions
$\phi _1$ and $\phi _2$ in $\mascS (\rr d)$, is defined by
the formula
\begin{equation}\label{Toepdef}
(\tp _{\phi _1,\phi _2}(a)f_1,f_2)_{L^2(\rr d)} = 
(a(2\cdo )W_{f_1, {\phi _1}},W_{f_2, {\phi _2}})_{L^2(\rr {2d})}
\end{equation}
when $f_1,f_2\in \mascS (\rr d)$.

\par

The definition of $\tp
_{\phi _1,\phi _2}(a)$ extends in several ways
(cf.~ e.{\,}g.~ \cite{AbdCorTof,CorGro1,HeWon,Toft3} and the references
therein).
In several of these extensions, we
interpret Toeplitz operators as pseudo-differential operators, using
the fact that
\begin{equation}\label{Eq:ToepWeyl}
\begin{aligned}
\operatorname{Tp}_{\phi _1,\phi _2}(a) &= \op _A(a\ast u)
\quad \text{when}
\\[1ex]
u(X) &= (2\pi)^{-\frac d2}W_{\phi _2,\phi _1}^A(-X).
\end{aligned}
\end{equation}
This formula holds when $\phi _1,\phi _2$ are suitable window functions on $\rr d$, for any 
$A\in \GL (d,\mathbf R)$ and appropriate distribution
$a$ on $\rr {2d}$. Here $*$ is the
usual convolution between Schwartz functions and tempered
distributions, and here we also use
capitals
$$
X=(x,\xi )\in \rr {2d},\quad Y=(y,\eta )\in \rr {2d},
\qquad x,y,\xi ,\eta \in \rr d,
$$
to denote elements in the phase space $\rr {2d}$.

\par

The relation \eqref{Eq:ToepWeyl} is well-known in the Weyl case
(see e.{\,}g. \cite{Fol}), and when $A=tI_d$ for some
$t\in \mathbf R$
(cf. e.{\,}g. \cite[(1.23)]{Toft13} and the references therein). For general
$A\in \GL (d,\mathbf R)$, \eqref{Eq:ToepWeyl} is an immediate
consequence of the case $A=tI_d$, \eqref{Eq:CalculiTransform}
\eqref{Eq:WignerDistRelDiffMat}, and the fact that
$$
e^{i\scal {AD_{\xi}}{D_x }} (a\ast u) =
a * (e^{i\scal {AD_{\xi}}{D_x }} u),
$$
which follows via integration by parts.

\par

\subsection{Schatten-von Neumann classes}
In what follows we recall some facts on Schatten-von Neumann operators,
given in \cite{Sim}. Let $T$ be a linear and continuous map from the
Hilbert space $\maclH _1$
into the Hilbert space $\maclH _2$, {and let $j\ge 1$ be an integer.
Also let
$\mascI _{0,j}(\maclH _1,\maclH _2)$ be the set of all linear and
continuous operators from $\maclH _1$ to $\maclH _2$ with rank
at most $j-1$.
The \emph{singular value} of $T$ of order $j$
is defined by
$$
\sigma _j(T)=\sigma _j(T;\maclH _1,\maclH_2)
\equiv
\inf _{T_0\in \mascI _{0,j}}\nm {T-T_0}{\maclH _1\to \maclH _2},
\qquad j\in \mathbf Z_+ .
$$
Evidently
$\sigma _j(T;\maclH _1,\maclH_2)$ decreases with $j$, and
$\sigma _1(T;\maclH _1,\maclH_2)$ is equal to the operator
norm $\nm T{\maclH _1\to \maclH _2}$ of $T$.

\par

Throughout the paper, all Hilbert spaces are assumed to be separable,
and observe that this is always the case for Hilbert spaces which
are continuously embedded in $\mascS '(\rr d)$,
in view of \cite[Proposition 1.2]{RaToVi}.
However, we note that most parts of what is described here also
hold when $\maclH _1$ and $\maclH _2$ are allowed to be non-separable.

\par

In the following definition we present a broad family of
Schatten-von Neumann classes.

\par

\begin{defn}\label{Def:GeneralSchattenClasses}
Let $\maclH _1$, $\maclH_2$ be Hilbert spaces,
$T$ be a linear operator from $\maclH _1$ to $\maclH _2$,
and let $\maclB \subseteq \ell _0'(\mathbf Z_+)$ be a quasi-Banach space.
\begin{enumerate}
\item The $\maclB$ Schatten-von Neumann (quasi-)norm of $T$ is given by
$$
\nm T{\mascI _\maclB}=\nm T{\mascI _\maclB (\maclH _1,\maclH _2)}
\equiv
\nm { \{ \sigma _j(T;\maclH_1,\maclH_2)\} _{j=1}^\infty}{\maclB}.
$$

\vrum

\item The $\maclB$ Schatten-von Neumann class
$\mascI _\maclB = \mascI _\maclB (\maclH _1,\maclH _2)$
consists of all linear and continuous operators $T$ from $\maclH _1$
to $\maclH _2$ such that
$\nm T{\mascI _{\maclB} (\maclH _1,\maclH _2)}$ is finite. 
The topology
of $\mascI _\maclB = \mascI _\maclB (\maclH _1,\maclH _2)$
is given through the (quasi-)norm
$\nm \cdo {\mascI _\maclB (\maclH _1,\maclH _2)}$.
\end{enumerate}
%
%
\end{defn}

\par

\begin{defn}
\label{Def:OrliczSchattenClasses}
Let $\Phi$ be a quasi-Young function, $p\in (0,\infty ]$
and let $\maclH _1$ and $\maclH _2$ be Hilbert spaces.
Then let
$$
\mascI _\Phi = \mascI _{\ell ^\Phi},
\quad
\mascI _p = \mascI _{\ell ^p}
\quad \text{and}\quad
\mascI _\sharp = \mascI _{\ell ^\sharp}.
$$
\begin{itemize}
\item
The space $\mascI _\Phi (\maclH _1,\maclH _2)$
is called the Orlicz Schatten-von Neumann class with respect to
$\Phi$, $\maclH _1$ and $\maclH _2$, or the $\Phi$-Schatten class.

\vrum

\item
The  space $\mascI _p (\maclH _1,\maclH _2)$ is called the
(classical) Schatten-von Neumann class with respect to 
$p$, $\maclH _1$ and $\maclH _2$, or the $p$-Schatten class.
\end{itemize}
\end{defn}

\par

%

We observe that $\mascI _p(\maclH _1,\maclH _2)$ increases with
$p$, and that $\mascI _\Phi (\maclH _1,\maclH _2)$ decreases with $\Phi$.
In fact, for the latter conclusion, it suffices to detect the decreasing property
with respect to $\Phi$
near origin, which is shown in the following proposition.
The result follows from Proposition
\ref{Prop:DiscreteNorms1} and the definitions.
The details are left for the reader.

\par

\begin{prop}\label{Prop:SchattenEmbed}
Let $\maclH _1$ and $\maclH _2$ be Hilbert spaces, and let
$\Phi _1$ and $\Phi _2$ be quasi-Young functions such that for some
$T>0$ it holds
$$
\Phi _2(t)\lesssim \Phi _1(t),
\quad \text{when}\quad
t\in (0,T].
$$
Then
$\mascI _{\Phi _1}(\maclH _1,\maclH _2)
\hookrightarrow
\mascI _{\Phi _2}(\maclH _1,\maclH _2)$.
\end{prop}

\par

We let $\mascI _0(\maclH _1,\maclH _2)$ be the
space of all linear operators from $\maclH_ 1$ to
$\maclH _2$ with finite rank. Then
$$
\mascI _0(\maclH _1,\maclH _2)
=
\bigcup _{j=1}^\infty \mascI _{0,j}(\maclH _1,\maclH _2),
$$
It follows that
$\mascI _\sharp = \mascI _\sharp (\maclH _1,\maclH _2)$
consists of all $T\in \mascI _\infty (\maclH _1, \maclH _2)$
such that
$$
\{\sigma _j(T;\maclH _1,\maclH _2)\} _{j=1}^\infty \in \ell ^\sharp (\mathbf Z_+),
$$
with topology is given through the norm
$\nm \cdo {\mascI _\infty (\maclH _1,\maclH _2)}$.

\par

We notice that
$$
\mascI _1(\maclH _1,\maclH _2),
\quad
\mascI _2(\maclH _1,\maclH _2),
\quad
\mascI _\sharp (\maclH _1,\maclH _2),
\quad \text{and}\quad
\mascI _\infty(\maclH _1,\maclH _2),
$$
are the spaces of trace-class, Hilbert-Schmidt, compact, and linear
and continuous operators from $\maclH _1$ to $\maclH _2$, respectively,
also in norms.

\par

We set
$$
\mascI _{\maclB}(\maclH ) = \mascI _{\maclB}(\maclH ,\maclH),
\quad
\mascI _p(\maclH ) = \mascI _p(\maclH ,\maclH)
\quad \text{and}\quad
\mascI _\Phi (\maclH ) = \mascI _\Phi (\maclH ,\maclH)
$$
when $\maclH$ is a Hilbert space.
By straight-forward application of the spectral theorem, it follows
that the definition of $\mascI _p(\maclH )$, $\mascI _\Phi (\maclH )$
and their norms
in Definition \ref{Def:OrliczSchattenClasses} agree
with the definitions of those quantities in the introduction.
In fact, we have the following.

\par

\begin{prop}
\label{Prop:OrlSchattTop}
Let $\maclH _1$ and $\maclH _2$ be Hilbert spaces,
and $\Phi$ be a quasi-Young function of order $r\in (0,1]$.
Then $\mascI _\Phi (\maclH _1,\maclH _2)$ is a quasi-Banach
space of order $r$, and
\begin{equation}
\label{Eq:AltOrlSchattNorm}
\nm T{\mascI _\Phi}
=
\sup \nm { \{ (Tf_j,g_j)_{\maclH _2}\}_{j=1}^\infty}{\ell ^\Phi (\mathbf Z_+)},
\quad
T\in \mascI _\infty (\maclH _1,\maclH _2).
\end{equation}
Here the supremum is taken over all orthonormal sequences
$\{ f_j\} _{j=1}^\infty \in \ON (\maclH _1)$
and
$\{ g_j\} _{j=1}^\infty \in \ON (\maclH _2)$.
\end{prop}

\par

\begin{proof}
If $\Phi$ is a Young function, then the identity
\eqref{Eq:AltOrlSchattNorm} follows from  \cite[Proposition 2.6]{Sim}.
For general quasi-Young functions $\Phi$,
\eqref{Eq:AltOrlSchattNorm} follows as well from  \cite[Proposition 2.6]{Sim},
and its proof. The details are left for the reader.
%

\par

If $T_1,T_2\in \mascI _\infty (\maclH _1,\maclH _2)$, then
\eqref{Eq:AltOrlSchattNorm} gives
\begin{align*}
\nm {T_1+T_2}{\mascI _\Phi}^r
&=
\sup \nm { \{ ((T_1+T_2)f_j,g_j)_{\maclH _2}\}_{j=1}^\infty}{\ell ^\Phi}^r
\\[1ex]
&\le
\sup \left (
\nm { \{ (T_1f_j,g_j)_{\maclH _2}\}_{j=1}^\infty}{\ell ^\Phi}^r
+
\nm { \{ (T_2f_j,g_j)_{\maclH _2}\}_{j=1}^\infty}{\ell ^\Phi}^r
\right )
\\[1ex]
&\le
\sup \nm { \{ (T_1f_j,g_j)_{\maclH _2}\}_{j=1}^\infty}{\ell ^\Phi}^r
+
\sup \nm { \{ (T_2f_j,g_j)_{\maclH _2}\}_{j=1}^\infty}{\ell ^\Phi}^r
\\[1ex]
&=
\nm {T_1}{\mascI _\Phi}^r + \nm {T_2}{\mascI _\Phi}^r,
\end{align*}
which gives the result.
\end{proof}

\par

\begin{rem}
Let $\maclH _1$ and $\maclH _2$ be Hilbert spaces, and let
$\maclB$ be a quasi-Banach space of sequences
on $\mathbf Z_+$ of order $r\in (0,1]$.
By similar arguments as in the proof of
Proposition \ref{Prop:OrlSchattTop} we see that
$\mascI _{\maclB} (\maclH _1,\maclH _2)$ is a quasi-Banach
space of the same order as $\maclB$, and
\begin{equation}
\nm T{\mascI _{\maclB}} = \sup \nm { \{ (Tf_j,g_j)_{\maclH _2}\}_{j=1}^\infty}{\maclB},
\quad
T\in \mascI _\infty (\maclH _1,\maclH _2).
\end{equation}
Here the supremum is again taken over all orthonormal sequences
$\{ f_j\} _{j=1}^\infty \in \ON (\maclH _1)$
and
$\{ g_j\} _{j=1}^\infty \in \ON (\maclH _2)$.
(See also
\cite[Proposition 2.6]{Sim},
and its proof.)

\par

Suppose additionally
that the following holds true:
\begin{enumerate}
\item $\maclB$ is norm invariant under permutations, i.{\,}e.,
$$
\nm a{\maclB}=\nm b{\maclB},
$$ 
for all sequences $\{ a_j\} _{j=1}^\infty$ and $\{ b_j\} _{j=1}^\infty$
such that $b_j=a_{\sigma (j)}$, $j\in \mathbf Z_+$, for some permutation
$\sigma$ on $\mathbf Z_+$;

\vrum

\item $e_{j_0}\equiv
\{ \delta _{j,j_0}\} _{j=1}^\infty \in \maclB$ for some $j_0\in \mathbf Z_+$.
\end{enumerate}
Then $\mascI _r(\maclH _1,\maclH _2)\subseteq \mascI _{\maclB}(\maclH _1,\maclH _2)$,
and
\begin{equation}
\label{Eq:BSchattClassSchattEst}
\nm T{\mascI _{\maclB}(\maclH _1,\maclH _2)}
\le
c\nm T{\mascI _r(\maclH _1,\maclH _2)},
\quad
c=\nm {e_{j_0}}{\maclB}.
\end{equation}

\par

In fact, by a combination of (1) and (2) it follows that
$$
c=\nm {e_{j_0}}{\maclB}
$$
for \emph{every} $j_0\in \mathbf Z_+$. This gives
\begin{align*}
\nm T{\mascI _{\maclB}}^r
&=
\nm {\{ \sigma _j(T)\} _{j=1}^\infty}{\maclB}^r
=
\NM {\sum_{j=1}^\infty \sigma _j(T)e_j}{\maclB}^r
\\[1ex]
&\le
\sum _{j=1}^\infty \sigma _j(T)^r\nm {e_j}{\maclB}^r
=
c^r\sum _{j=1}^\infty \sigma _j(T)^r
=
c^r\nm T{\mascI _r}^r,
\end{align*}
and the assertion follows.
\end{rem}

\par

We observe that (1) and (2) hold true when
$\maclB = \ell ^\Phi (\mathbf Z_+)$ for some quasi-Young function $\Phi$.
Hence \eqref{Eq:BSchattClassSchattEst} gives
\begin{equation}
\tag*{(\ref{Eq:BSchattClassSchattEst})$'$}
\label{Eq:BSchattClassSchattEstOrl}
\nm T{\mascI _{\Phi}(\maclH _1,\maclH _2)}
\le
\frac 1{\Phi ^{-1}(1)}\nm T{\mascI _r(\maclH _1,\maclH _2)},
\end{equation}
provided the order of $\Phi$ is equal to $r$.

\par

By Proposition \ref{Prop:OrlSchattTop} it follows that
$$
\mascI _\Phi (\maclH _1,\maclH _2),
\quad
\mascI _p (\maclH _1,\maclH _2)
\quad \text{and}\quad
\mascI _\sharp (\maclH _1,\maclH _2)
$$
are Banach spaces when $\Phi$ is a Young function
and $p\in [1,\infty ]$.

\par

Properties valid for both $\mascI _\infty (\maclH _1,\maclH _2)$ 
and $\mascI _{0,j} (\maclH _1,\maclH _2)$ sometimes impose
similar properties for Schatten-von Neumann classes, which is indicated
in the following proposition. 

\par

\begin{prop}
\label{Prop:ContSchattInvariance}
Let $\maclH _1$ and $\maclH _2$ be Hilbert spaces, and let
$\maclB$ be a quasi-Banach space of sequences on $\mathbf Z_+$.
Suppose that $\maclL$ is a linear map on $\mascI _\infty (\maclH _1,\maclH _2)$
such that the following conditions are fulfilled:
\begin{itemize}
\item $\maclL$ is a bijective isometry on $\mascI _\infty (\maclH _1,\maclH _2)$;

\vrum

\item $\maclL$ restricts to a bijection on $\mascI _{0,j} (\maclH _1,\maclH _2)$
for every $j\in \mathbf Z_+$.
\end{itemize}
Then the following is true:
\begin{enumerate}
\item if $T\in \mascI _\infty (\maclH _1,\maclH _2)$ and $j\in \mathbf Z_+$,
then $\sigma _j(\maclL (T))=\sigma _j(T)$;

\vrum

\item $\maclL$ restricts to a bijective isometry on
$\mascI _{\maclB} (\maclH _1,\maclH _2)$.
\end{enumerate}
\end{prop}

\par

\begin{proof}
For any $j\in \mathbf Z_+$ we have
\begin{align*}
\sigma _j(\maclL (T))
&=
\inf _{T_0\in \mascI _{0,j}}\nm {\maclL (T)-T_0}{\mascI _\infty}
=
\inf _{T_0\in \mascI _{0,j}}\nm {\maclL (T-T_0)}{\mascI _\infty}
\\[1ex]
&=
\inf _{T_0\in \mascI _{0,j}}\nm {T-T_0}{\mascI _\infty}
=
\sigma _j(T),
\end{align*}
and (1) follows.

\par

The assertion (2) follows from the definitions and (1).
\end{proof}

%

We observe that H{\"o}lder's inequality holds for compositions. That is,
if the quasi-Young functions $\Phi _k$ satisfy
\begin{equation}\label{Eq:HolderYoungFuncCond}
\Phi _0(t_1t_2) \le \Phi _1(t_1) +\Phi _2(t_2),
\qquad
t_1,t_2\in [0,t_0],
\end{equation}
for some $t_0>0$, and
\begin{equation}\label{Eq:TkDefs}
T_k\in \mascI _{\Phi _k}(\maclH _{k-1},\maclH _k),\quad k=1,2,
\end{equation}
for some Hilbert spaces $\maclH _k$, $k=0,1,2$, then 
\begin{equation}\label{Eq:HolderOrliczComp}
\begin{gathered}
T_2\circ T_1 \in \mascI _{\Phi _0}(\maclH _0,\maclH _2)
\\[1ex]
\text{and}\quad
\nm {T_2\circ T_1}{\mascI _{\Phi _0}(\maclH _0,\maclH _2)}
\le
2\nm {T_1}{\mascI _{\Phi _1}(\maclH _0,\maclH _1)}
\nm {T_2}{\mascI _{\Phi _2}(\maclH _1,\maclH _2)}.
\end{gathered}
\end{equation}
(See e.{\,}g. \cite[3.3, Lemma 6]{RaoRen1} in combination with
\cite[Theorem 2.8]{Sim} and their proofs.)
Another condition
on the quasi-Young functions $\Phi _0$, $\Phi _1$ and $\Phi _2$
to guarantee that \eqref{Eq:HolderOrliczComp} should hold is
\begin{equation}\label{Eq:HolderYoungFuncCond2}
\Phi _1^{-1}(s)\Phi _2^{-1}(s) \le \Phi _0^{-1}(s).
\end{equation}
Here we observe that the condition
\eqref{Eq:HolderYoungFuncCond2}
implies \eqref{Eq:HolderYoungFuncCond}.
For the reader's convenience we have included a proof of this
fact in Appendix \ref{App:B}
(see Proposition \ref{Prop:HolderYoungFuncCondEsts}).

%
%
%
%
%
%

\par

%

By choosing the quasi-Young functions as in Remark \ref{Rem:PhiLeb},
where the involved Lebesgue exponents should satisfy the
usual H{\"o}lder condition
\eqref{Eq:HolderYoungFuncCond}--\eqref{Eq:HolderOrliczComp},
we may avoid the factor $2$ in \eqref{Eq:HolderOrliczComp},
giving the usual H{\"o}lder's inequality for operator compositions.
That is, we have
$$
T_2\circ T_1 \in \mascI _{p_0}(\maclH _0,\maclH _2)
$$
and 
$$
\nm {T_2\circ T_1}{\mascI _{p_0}(\maclH _0,\maclH _2)}
\le
\nm {T_1}{\mascI _{p_1}(\maclH _0,\maclH _1)}
\nm {T_2}{\mascI _{p_2}(\maclH _1,\maclH _2)}, 
$$
when
$$
p_0,p_1,p_2\in (0,\infty ],
\quad
\frac 1{p_1}+\frac 1{p_2} = \frac 1{p_0}
\quad \text{and}\quad
T_k\in \mascI _{p_k}(\maclH _{k-1},\maclH _k),\ k=1,2.
$$

\par

For the adjoint $T^*$ and the absolute value $|T|\equiv (T^*\circ T)^{\frac 12}$
of $T\in \mascI _\infty (\maclH _1,\maclH _2)$, we have
$$
T\in \mascI _\Phi (\maclH _1,\maclH _2)
\quad \Leftrightarrow \quad
T^*\in \mascI _\Phi (\maclH _2,\maclH _1)
\quad \Leftrightarrow \quad
|T|\in \mascI _\Phi (\maclH _1),
$$
when $\Phi$ is a quasi-Young function. Moreover,
$$
\nm T{\mascI _\Phi (\maclH _1,\maclH _2)}
=
\nm {T^*}{\mascI _\Phi (\maclH _2,\maclH _1)}
=
\nm {\, |T|\, }{\mascI _\Phi (\maclH _1)}.
$$

\par

We recall that for any trace-class operator $T\in \mascI _1(\maclH )$
acting on the Hilbert space $\maclH$, the trace
$$
\Tr (T) = \Tr _{\maclH}(T) \equiv \sum _{j=1}^\infty (Tf_j,f_j)_\maclH
$$
is well-defined. The trace is
independent of the choice of orthonormal basis $\{ f_j\} _{j=1}^\infty$ in $\maclH$,
and satisfies
$$
|\Tr _{\maclH}(T)| \le \nm T{\mascI _1(\maclH)},
$$
with equality, if and only if there is a non-zero constant $c\in \mathbf C$ such that
$cT$ is a positive semi-definite operator. Hence,
if $\Phi$ is a Young function, then \eqref{Eq:HolderOrliczComp}
shows that
$$
(T_1,T_2)_{\mascI _2} \equiv \Tr _{\maclH _1}(T_2^*\circ T_1)
$$
is well-defined when $T_1\in \mascI _{\Phi}(\maclH _1,\maclH _2)$
and $T_2\in \mascI _{\Phi ^*}(\maclH _1,\maclH _2)$, and
\begin{equation}
\label{Eq:OpScalProd}
|(T_1,T_2)_{\mascI _2}|
\le
2\nm {T_1}{\mascI _{\Phi}(\maclH _1,\maclH _2)}
\nm {T_2}{\mascI _{\Phi ^*}(\maclH _1,\maclH _2)}.
\end{equation}
If $p,p'\in [1,\infty ]$, $T_1\in \mascI _{p}(\maclH _1,\maclH _2)$
and $T_2\in \mascI _{p'}(\maclH _1,\maclH _2)$, then
\eqref{Eq:OpScalProd} is improved into
$$
|(T_1,T_2)_{\mascI _2}|
\le
\nm {T_1}{\mascI _{p}(\maclH _1,\maclH _2)}
\nm {T_2}{\mascI _{p'}(\maclH _1,\maclH _2)}.
$$
We also remark that $\mascI _2(\maclH _1,\maclH _2)$
is a Hilbert space with the scalar product $(\cdo ,\cdo )_{\mascI _2}$
(cf. Remark \ref{Rem:PhiLeb}).


\par

The following two propositions are straight-forward consequences of 
the spectral theorem, and Propositions \ref{Prop:OrlDuality}
and \ref{Prop:DiscrNmLocDelCond1}.
(See e.{\,}g. \cite{Sim}.) Here and in what
follows we identify operators with their kernels, and we let $\ON (\maclH)$
be the set of all orthonormal sets in the Hilbert space $\maclH$.

\par

\begin{prop}\label{Prop:SpectralComp}
Let $T$ be a linear operator from the Hilbert space $\maclH _1$ to
the Hilbert space $\maclH _2$.
Then $T\in \mascI _\sharp (\maclH _1,\maclH _2)$,
if and only if there is a non-negative decreasing sequence
$\lambda =\{ \lambda _j\} _{j=1}^\infty \in \ell ^\sharp (\mathbf Z_+)$,
$\{ f_{k,j} \} _{j=1}^\infty \in \ON (\maclH _k)$, $k=1,2$, and such that
\begin{equation}\label{Eq:SpectExp0}
T=\sum _{j=1}^\infty \lambda _jf_{2,j}\otimes f_{1,j},
\qquad
\nm T{\mascI _\sharp (\maclH _1,\maclH _2)} = \nm \lambda{\ell ^\sharp}
\text .
\end{equation}
\end{prop}

\par

\begin{prop}
\label{Prop:Spectral0}
Let $T$ be a linear operator from the Hilbert space $\maclH _1$ to
the Hilbert space $\maclH _2$, and let $\Phi$ be a quasi-Young function.
\begin{enumerate}
\item 
If $\Phi$ is positive, then
$\mascI _\Phi (\maclH _1,\maclH _2)$ is continuously embedded in
$\mascI _\sharp (\maclH _1,\maclH _2)$, and
$T\in \mascI _\Phi (\maclH _1,\maclH _2)$, if and only if
$\lambda$ in  \eqref{Eq:SpectExp0}
satisfies $\nm {\lambda}{\ell ^\Phi}<\infty$, and
\begin{equation}\label{Eq:SpectEstOrlA}
\nm T{\mascI _\Phi (\maclH _1,\maclH _2)} = \nm \lambda{\ell ^\Phi}.
\end{equation}

\vrum

\item If $\Phi$ is not positive, then
$\mascI _\Phi (\maclH _1,\maclH _2)$ is equal to
$\mascI _\infty (\maclH _1,\maclH _2)$ with equivalent (quasi-)norms.

\vrum

\item If $\Phi$ satisfies a local $\Delta _2$-condition, then
$\mascI _0(\maclH _1,\maclH _2)$ is dense in
$\mascI _\Phi (\maclH _1,\maclH _2)$
and in $\mascI _\sharp (\maclH _1,\maclH _2)$.

\vrum

\item If $\Phi$ satisfies a local $\Delta _2$-condition and in addition
is a Young function, then
the dual of $\mascI _\Phi (\maclH _1,\maclH _2)$ 
is equal to $\mascI _{\Phi ^*} (\maclH _1,\maclH _2)$, also in norms, 
through the form $(\cdo ,\cdo )_{\mascI _2}$.
\end{enumerate}
\end{prop}

\par

\begin{rem}
\label{Rem:SpectralFiniteRank}
Let $\maclH _1$ and $\maclH _2$ be Hilbert spaces.
Then it follows that $T\in \mascI _0(\maclH _1,\maclH _2)$
consists of all expansions \eqref{Eq:SpectExp0} such that
$\lambda \in \ell _0(\mathbf Z_+)$.
\end{rem}

\par

\begin{rem}
It follows that the joint statements of (1) and (2) in
Proposition \ref{Prop:Spectral0} are essentially equivalent to
the joint statements of the following.
\begin{enumerate}
\item[(1)$'$]
\emph{
$\mascI _\Phi (\maclH _1,\maclH _2)$ is continuously embedded in
$\mascI _\sharp (\maclH _1,\maclH _2)$, if and only if $\Phi$ is positive.}

\vrum

\item[(2)$'$]
\emph{
$\mascI _\Phi (\maclH _1,\maclH _2)$ is equal to
$\mascI _\infty (\maclH _1,\maclH _2)$ with equivalent (quasi-)norms,
if and only if $\Phi$ is \emph{not} positive.}
\end{enumerate}
\end{rem}

\par

We finish the section by recalling the following
special case of \cite[Theorem 6.1]{CheSigTof},
concerning operators with kernels in $\mascS (\rr {2d})$ when
acting on $L^2(\rr {d})$. Here $\mathbf N =\{ 0,1,2,\dots \}$
is the set of all natural numbers.

\par

\begin{prop}
\label{Prop:SchwKerOpProp}
Let $\Phi$ be a quasi-Young function. Then the map $K\mapsto T_K$
from $\mascS (\rr {2d})$ to $\mascI _\Phi (L^2(\rr d))$ is a continuous
injection. Furthermore, if $K\in \mascS (\rr {2d})$ and $T=T_K$
has the expansion \eqref{Eq:SpectExp0} for some
$\{ f_{k,j}\} _{j=1}^\infty \in \ON (\maclH _k)$, then
the following is true:
\begin{enumerate}
\item $\{ j^N\lambda _j \} _{j=1}^\infty \in \ell ^\infty (\mathbf Z_+)$, for every
$N\ge 0$;

\vrum

\item if $\lambda _j>0$, then $f_{1,j},f_{2,j}\in \mascS (\rr d)$ and 
$$
\{ j^N\nm {x^\alpha D^\beta f_{k,j}}{L^\infty} \} _{j=1}^\infty \in \ell ^\infty (\mathbf Z_+),
\qquad
N\ge 0,\ k=1,2,\ \alpha ,\beta \in \nn d.
$$
\end{enumerate}
\end{prop}

\par

\section{Schatten symbol classes}
\label{sec2}

\par

In this section we introduce general Schatten-von Neumann
symbol classes with respect to quasi-Banach spaces $\maclB$
of sequences, as spaces of distributions whose corresponding
pseudo-differential operators should belong to Schatten-von Neumann
classes with respect to $\maclB$ on $L^2(\rr d)$.
By choosing these quasi-Banach spaces in suitable ways, we
obtain analogous Orlicz Schatten-symbol classes.
Thereafter we show some of properties for Orlicz
Schatten-symbol classes, e.{\,}g.~ spectral
resolutions, and their invariance under translations and modulations.

\par

%

We begin with defining symbol classes which correspond to
Schatten-von Neumann operators acting on $L^2(\rr d)$.

\par

\begin{defn}\label{Def:PseudoSchattenClasses}
Let $\maclB \subseteq \ell _0'(\mathbf Z_+)$ be a quasi-Banach space,
and let $A\in \GL (d,\mathbf R)$.
\begin{itemize}
\item $s_{A,\maclB}(\rr {2d})$ consists of all $a\in \mascS '(\rr {2d})$
such that $\op _A(a)\in \mascI _{\maclB} (L^2(\rr d))$. The
topology of $s_{A,\maclB}(\rr {2d})$ is given through the quasi-norm
$$
\nm {a}{s_{A,\maclB}}\equiv (2\pi )^{\frac d2}\nm {\op_A(a)}{\mascI _{\maclB} },
\qquad a\in \mascS '(\rr {2d}).
$$

\vrum

\item $s_{A,0}(\rr {2d})$ consists of all $a\in \mascS '(\rr {2d})$
such that $\op _A(a)\in \mascI _0 (L^2(\rr d))$.

\vrum

\item $s_{A,\flat}(\rr {2d})=s_{A,0}(\rr {2d})\cap \mascS (\rr {2d})$.
\end{itemize}
\end{defn}

\par

If $\Phi$ is a quasi-Young function, $p\in (0,\infty ]$,
$\Phi _{[p]}$ is the same as in Remark \ref{Rem:PhiLeb},
then let
\begin{align*}
s_{A,\Phi}(\rr {2d})
\equiv s_{A,\ell ^{\Phi}}(\rr {2d}),
\quad
s_{A,p}(\rr {2d})&\equiv s_{A,\Phi _{[p]}}(\rr {2d})
\\[1ex]
\text{and}\qquad
s_{A,\sharp}(\rr {2d})&\equiv s_{A,\ell ^{\sharp}}(\rr {2d}),
\end{align*}
with corresponding norms given by
$$
\nm \cdo {s_{A,\Phi}}\equiv \nm \cdo {s_{A,\ell ^\Phi}},
\quad
\nm \cdo {s_{A,p}}\equiv \nm \cdo {s_{A,\Phi _{[p]}}}
\quad \text{and}\quad
\nm \cdo {s_{A,\sharp}}\equiv \nm \cdo {s_{A,\ell ^{\sharp}}},
$$
respectively.

\par

By the definitions and Lemma \ref{Lem:KernelPseudoDiffThm},
we obtain the following. The details are left for the reader.

\par

\begin{lemma}
\label{Lemma:SymbOpMap}
Let $\maclB \subseteq \ell _0'(\mathbf Z_+)$ be a quasi-Banach space,
and let $A\in \GL (d,\mathbf R)$. Then
\begin{equation}
\label{Eq:SymbOpMap}
a\mapsto (2\pi )^{\frac d2}\op _A(a)
\end{equation}
is an isometric isomorphism from $s_{A,\maclB} (\rr {2d})$ to $\mascI _{\maclB} (L^2(\rr d))$.
\end{lemma}

\par

In particular, it follows from the previous lemma
that the spaces  $s_{A,\Phi}(\rr {2d})$ and  $s_{A,p}(\rr {2d})$ in Definition 
\ref{Def:PseudoSchattenClasses}
are quasi-Banach spaces, since similar facts hold true for the spaces in
Definition \ref{Def:OrliczSchattenClasses}. If in addition $\Phi$ is a Young function
and $p\in [1,\infty ]$, then $s_{A,\Phi}(\rr {2d})$ and $s_{A,p}(\rr {2d})$
are Banach spaces.
Also notice that $s_{A,\sharp}(\rr {2d})$ is a Banach space.

\par

Since the mappings \eqref{atkernel} and \eqref{homeoF2tmap}
are homeomorphisms on $\mascS (\rr {2d})$, the following
result follows from Proposition \ref{Prop:SchwKerOpProp}.
The details are left for the reader. (See also \cite{Fol}.)

\par

\begin{prop}
Let $\Phi$ be a quasi-Young function and $A\in \GL (d,\mathbf R)$. Then
$\mascS (\rr {2d})\hookrightarrow s_{A,\Phi}(\rr {2d})$.
\end{prop}

\par

If $\Phi$ is a Young function, then let
\begin{equation}
\label{Eq:SymbDual}
\begin{aligned}
(a,b)_{s_{A,2}}
&\equiv
(2\pi )^{d} (\op _A(a),\op _A(b))_{\mascI _2},
\\[1ex]
a &\in s_{A,\Phi} (\rr {2d}),\ b\in s_{A,\Phi ^*} (\rr {2d}).
\end{aligned}
\end{equation}
By a straight-forward application of Fourier's inversion formula
we obtain $s_{A,2}(\rr {2d})=L^2(\rr {2d})$, with equality in norms
and scalar products. Due to \eqref{Eq:SymbDual},
we therefore define
$$
(a,b)_{L^2}\equiv (a,b)_{s_{A,2}},
\quad \text{when}\quad
a \in s_{A,\Phi} (\rr {2d}),\ b\in s_{A,\Phi ^*} (\rr {2d}).
$$

\par

Since the Weyl case is especially important we also set
$$
s_{\maclB}^w=s_{A,\maclB},
\quad
s_\Phi ^w=s_{A,\Phi},
\quad
s_p ^w=s_{A,p},
\quad
s_0^w=s_{A,0},
\quad \text{and}\quad
\quad
s_\flat ^w=s_{A,\flat},
$$
when $A= {\textstyle{\frac{1}{2}}} \cdot I_d$,
$\maclB \subseteq \ell _0'(\mathbf Z_+)$ is a quasi-Banach space,
$\Phi$ is a quasi-Young function, and $p\in (0,\infty ]$.

\par

The assertions
(2) and (3) in the following proposition follow by combining
\eqref{Eq:AWigPseudoLink2}, \eqref{Eq:HolderOrliczComp},
Proposition \ref{Prop:Spectral0}, and Lemma
\ref{Lemma:SymbOpMap}.
The details are left for the reader.
Here and in what follows we let $\check f(x)=f(-x)$,
for any $f\in \mascS '(\rr d)$ and $x\in \rr d$.

\par

\begin{prop}\label{Prop:Spectral0Symb}
Let $a\in \mascS '(\rr {2d})$, 
$\maclB \subseteq \ell _0'(\mathbf Z_+)$ be a quasi-Banach space,
$\Phi$ be a quasi-Young function,
and let $A\in \GL (d,\mathbf R)$.
\begin{enumerate}
\item If $B=I-A$, then
$$
a\in s_{A,\maclB}(\rr {2d})
\quad \Leftrightarrow \quad
\check a\in s_{A,\maclB}(\rr {2d})
\quad \Leftrightarrow \quad
\overline a\in s_{B,\maclB}(\rr {2d}),
$$
and
\begin{equation}
\label{Eq:SymbCheckConjNorms}
\nm a{s_{A,\maclB}} = \nm {\check a}{s_{A,\maclB}}
=
\nm {\overline a}{s_{B,\maclB}}\text .
\end{equation}

\vrum

\item $a\in s_{A,\sharp} (\rr {2d})$,
if and only if there is a non-negative decreasing sequence
$\lambda =\{ \lambda _j\} _{j=1}^\infty \in \ell ^\sharp (\mathbf Z_+)$,
$\{ f_{k,j} \} _{j=1}^\infty \in \ON (L^2(\rr d))$, $k=1,2$, and such that
\begin{equation}\label{Eq:SpectExp0Symb}
a=\sum _{j=1}^\infty \lambda _jW_{f_{2,j},f_{1,j}}^A,
\qquad
\nm a{s _{A,\sharp}}
=
\nm \lambda{\ell ^\sharp}
\text .
\end{equation}

\vrum

\item If in addition $\Phi$ is positive,
then $a\in s _{A,\Phi} (\rr {2d})$, if and only if
$\lambda$ in  \eqref{Eq:SpectExp0Symb}
satisfies $\nm {\lambda}{\ell ^\Phi}<\infty$, and
\begin{equation}\label{Eq:SpectEstOrlB}
\nm a{s_{A,\Phi}}
= 
\nm \lambda{\ell ^\Phi} .
\end{equation}
\end{enumerate}
\end{prop}

\par

\begin{proof}
Let $B=I-A$, $a\in s_{A,\infty}(\rr {2d})$, and let $K$  be the kernel of
$\op _A(a)$. Then it follows by straight-forward computations
that the kernels of $\op _A(\check a)$ and $\op _B( \overline a )$
are
$$
K_1(x,y)=K(-x,-y)
\quad \text{and}\quad
K_2(x,y)=\overline {K(y,x)},
$$
respectively. Especially it follows that the singular values of
$\op _A(a)$, $\op _A(\check a)$ and $\op _B( \overline a )$
agree. This gives (1).

\par

As remarked above, the assertions (2) and (3) follow from
\eqref{Eq:HolderOrliczComp}, Proposition
\ref{Prop:Spectral0} and Lemma \ref{Lemma:SymbOpMap}. 
\end{proof}

\par

\begin{cor}
\label{Cor:Spectral0Symb}
Let $a\in \mascS '(\rr {2d})$, $\Phi$ be a quasi-Young function,
and let $A\in \GL (d,\mathbf R)$. Then the following is true:
\begin{enumerate}
\item $a\in s_{A,0}(\rr {2d})$, if and only if \eqref{Eq:SpectExp0Symb}
holds for some non-negative decreasing sequence $\lambda \in \ell _0(\mathbf Z_+)$;

\vrum

\item  $a\in s_{A,\flat}(\rr {2d})$, if and only if \eqref{Eq:SpectExp0Symb}
holds for some non-negative decreasing sequence $\lambda \in \ell _0(\mathbf Z_+)$,
and in addition $f_{k,j}\in \mascS (\rr d)$, when $j\ge 1$, $k=1,2$ and
$\lambda _j\neq 0$.
\end{enumerate}
\end{cor}

\par

\begin{proof}
The assertion (1) follows from the definitions, Remark \ref{Rem:SpectralFiniteRank}
and Proposition \ref{Prop:Spectral0Symb}.

\par

Suppose that $a\in s_{A,\flat}(\rr {2d})$ has the expansion
\eqref{Eq:SpectExp0Symb}, and let $K$ be the kernel of the operator
$\op _A(a)$. Since the map $a\mapsto K_{a,A}$ in \eqref{atkernel}
is a homeomorphism on $\mascS (\rr {2d})$, and since
$\lambda \in \ell _0(\mathbf Z_+)$ in view of (1), it follows that
$$
K(x,y) = (2\pi)^{-\frac d2}\sum _{j=1}^N \lambda _j f_{2,j}(x)\overline{f_{1,j}(y)}
\in \mascS (\rr {2d}),
$$
when $N\ge 0$ is the rank of $\op _A(a)$.
For any $j_0\le N$, we obtain $\lambda _{j_0}>0$ and
\begin{align*}
\lambda _{j_0}f_{2,j_0}(x)
&=
\sum _{j=1}^N \lambda _{j}f_{2,j}(x)(f_{1,j_0},f_{1,j})_{L^2}
=
(2\pi )^{\frac d2}(K(x,\cdo ),f_{1,j_0})_{L^2}\in \mascS (\rr d),
\intertext{and}
\lambda _{j_0}f_{1,j_0}(y)
&=
\sum _{j=1}^N \lambda _{j}f_{1,j}(y)(f_{2,j_0},f_{2,j})_{L^2}
=
(2\pi )^{\frac d2}(f_{2,j_0},K(\cdo ,y))_{L^2}\in \mascS (\rr d),
\end{align*}
which gives the implication from the left to right in (2).

\par

On the  other hand, if $a$ has the expansion
\eqref{Eq:SpectExp0Symb}
for some $\lambda \in \ell _0(\mathbf Z_+)$,
and with $f_{k,j}\in \mascS (\rr d)$, when $j\ge 1$, $k=1,2$ and
$\lambda _j\neq 0$, then
$$
a\in s_{A,0}(\rr {2d})\cap \mascS (\rr {2d})=s_{A,\flat}(\rr {2d}).
$$
This gives the implication from the right to left in (2), and thereby the result.
\end{proof}

\par

\begin{rem}
\label{Rem:RankOne}
As an immediate consequence of
\eqref{Eq:SpectExp0Symb}
and
\eqref{Eq:SpectEstOrlB}, it follows that an operator
$\op _A(a)$ on $L^2(\rr d)$ is a rank-one operator, if and only if
its symbol $a$ is equal to $W_{f,g}^A$ for some $f,g\in L^2(\rr d)\setminus 0$,
and that
$$
\nm {W_{f,g}^A}{\mascI _\Phi} = \frac 1{\Phi ^{-1}(1)}\nm f{L^2}\nm g{L^2}
\quad \text{and}\quad
\nm {W_{f,g}^A}{\mascI _\sharp} = \nm f{L^2}\nm g{L^2},
$$
for any positive quasi-Young function $\Phi$.
\end{rem}

\par

The next proposition corresponds to
Propositions \ref{Prop:SchattenEmbed} and \ref{Prop:Spectral0}
on the level of symbols to pseudo-differential operators.

\par

\begin{prop}
\label{Prop:DiscreteNorms2}
Suppose that $\Phi$, $\Phi _1$ and $\Phi _2$ are
quasi-Young functions, and that $A\in \GL (d,\mathbf R)$.
\begin{enumerate}
\item If \eqref{Eq:LocYoungEst} is true for some $T>0$, then
$s_{A,\Phi _1}(\rr {2d})\hookrightarrow s_{A,\Phi _2}(\rr {2d})$.

\vrum

\item If $\Phi$ is positive, then
$s_{A,\Phi}(\rr {2d})\hookrightarrow s_{A,\sharp}(\rr {2d})$.

\vrum

\item If $\Phi$ is not positive, then
$s_{A,\Phi}(\rr {2d})=s_{A,\infty}(\rr {2d})$,
with equivalent (quasi-)norms.

\vrum

\item If $\Phi$ satisfies a local $\Delta _2$-condition, then
$s_{A,\flat}(\rr {2d})$ and $s_{A,0}(\rr {2d})$ are dense in $s_{A,\Phi}(\rr {2d})$
and in $s_{A,\sharp}(\rr {2d})$.

\vrum

\item If $\Phi$ satisfies a local $\Delta _2$-condition and in addition is
a Young function, then
the dual of $s_{A,\Phi}(\rr {2d})$ is equal to $s_{A,\Phi ^*}(\rr {2d})$,
through the $L^2$ form $(\cdo ,\cdo )_{L^2}$.
\end{enumerate}
\end{prop}

\par

\begin{proof}
The assertions (1), (2), (4) and (5) follows by a straight-forward
combination of 
Propositions \ref{Prop:SchattenEmbed}, \ref{Prop:Spectral0},
and Lemma \ref{Lemma:SymbOpMap}. From these results
it also follows that $s_{A,0}(\rr {2d})$ in (4) is dense in
$s_{A,\Phi}(\rr {2d})$ and in $s_{A,\sharp}(\rr {2d})$.

\par

It remains to prove the density properties of $s_{A,\flat}(\rr {2d})$ in (4).
Let $r\in (0,1]$ be the order of $\Phi$ and $a\in s_{A,0}(\rr {2d})$.
In view of \ref{Eq:BSchattClassSchattEstOrl},
it suffices to approximate $a$
with elements in $s_{A,\flat}(\rr {2d})$
in the $s_{A,r}(\rr {2d})$ norm.
However, this follows by a combination of Corollary
\ref{Cor:Spectral0Symb} with
the fact that $\mascS (\rr d)$ is dense in $L^2(\rr d)$, and suitable
estimates in Grahm-Schmidt's orthogonalization method. The
details are left for the reader.
\end{proof}

\par

\begin{rem}
\label{Rem:SmallSchattDenseSchwartz}
Additionally to (4) in Proposition \ref{Prop:DiscreteNorms2},
we recall that $s_{A,\flat}(\rr {2d})$ is dense in $\mascS (\rr {2d})$.
This is, for example, a straight-forward consequence of
\cite[Theorem 6.1]{CheSigTof}.
\end{rem}

\par

\begin{rem}
\label{Rem:SchattCalculTransf}
Let $A_1$ and $A_2$ be $d\times d$-matrices,
$a_1,a_2\in \mascS '(\rr {2d})$ be such that
\eqref{Eq:CalculiTransform} holds, and let
$\maclB \subseteq \ell _0'(\mathbf Z_+)$.
Then it follows
from \eqref{Eq:WignerDistRelDiffMat},
\eqref{Eq:AWigPseudoLink},
Proposition \ref{Prop:Spectral0Symb},
and Corollary \ref{Cor:Spectral0Symb} that
the following is true:
\begin{itemize}
\item $a_1\in s_{A_1,\maclB}(\rr {2d})$, if and only if
$a_2\in s_{A_2,\maclB}(\rr {2d})$, and
$$
\nm {a_1}{s_{A_1,\maclB}} = \nm {a_2}{s_{A_2,\maclB}} \text ;
$$
%
%

\vrum

\item $a_1\in s_{A_1,0}(\rr {2d})$, if and only if
$a_2\in s_{A_2,0}(\rr {2d})$;

\vrum

\item $a_1\in s_{A_1,\flat}(\rr {2d})$, if and only if
$a_2\in s_{A_2,\flat}(\rr {2d})$.
\end{itemize}
\end{rem}

\par

In some applications, we also need equivalent norms for Schatten-von Neumann
classes of Orlicz types, in analogous ways as in Proposition \ref{Prop:OrlDuality}
for ordinary Orlicz spaces. In fact, we have the following.

\par

\begin{prop}
\label{Prop:SchattSymbEquivNorm}
Let $\Phi$ be a Young function, $A\in \GL (d,\mathbf R)$, and let
\begin{equation}
\nmm a \equiv \sup |(a,b)_{L^2}|,
\end{equation}
where the supremum is taken over all $b\in s_{A,\Phi ^*}(\rr {2d})$ such that
$\nm b{s_{A,\Phi ^*}}\le 1$. Then
\begin{equation}
\nm a{s_{A,\Phi}} \le \nmm a \le 2\nm a{s_{A,\Phi}},
\qquad
a\in s_{A,\Phi}(\rr {2d}).
\end{equation}
\end{prop}

\par

The previous proposition is a straight-forward consequence
of Lemma \ref{Lemma:SymbOpMap} and the following.

\par

\begin{prop}
\label{Prop:SchattEquivNorm}
Let $\Phi$ be a Young function, $\maclH _1$ and $\maclH _2$
be Hilbert spaces, and let
\begin{equation}
\nmm T \equiv \sup |(T,S)_{\mascI _2}|,
\end{equation}
where the supremum is taken over all
$S\in \mascI _{\Phi ^*}(\maclH _1,\maclH _2)$ such that
$\nm S{\mascI _{\Phi ^*}}\le 1$. Then
\begin{equation}
\label{Eq:AltNormSchatEst}
\nm T{\mascI _{\Phi}} \le \nmm T \le 2\nm T{{\mascI _{\Phi}}},
\qquad
T\in \mascI _{\Phi}(\maclH _1,\maclH _2).
\end{equation}
\end{prop}

\par

\begin{proof}
Let $T\in \mascI _{\Phi}(\maclH _1,\maclH _2)$. The second
inequality in \eqref{Eq:AltNormSchatEst} follows from \eqref{Eq:OpScalProd}.

\par

In order to prove the first inequality in \eqref{Eq:AltNormSchatEst} we first assume that
$T$ is compact. Then $T$ is given by
\eqref{Eq:SpectExp0}, for some $\{ f_{k,j} \} _{j=1}^\infty \in \ON (\maclH _k)$, $k=1,2$,
and some non-negative non-decreasing sequence
$\lambda =\{ \lambda _j\} _{j=1}^\infty \in \ell ^\Phi (\mathbf Z_+)$.

\par

Let $B$ be the unit ball in $\mascI _{\Phi ^*}(\maclH _1,\maclH _2)$, and let
$B_T$ be the set of all elements $S$ with expansions
\begin{equation}
\label{Eq:DualExp}
S=\sum _{j=1}^\infty \mu _jf_{2,j}\otimes f_{1,j},
\qquad
\nm \mu{\ell ^{\Phi ^*}}\le 1.
\end{equation}
It follows that $B_T\subseteq B$, because
$\nm S{\mascI _{\Phi ^*}}=\nm \mu{\ell ^{\Phi ^*}}\le 1$ for $S$ in
\eqref{Eq:DualExp}.

\par

By Proposition \ref{Prop:OrlDuality} we obtain
\begin{equation}
\begin{aligned}
\label{Eq:EquivSchatNormRevEst}
\nm T{\mascI _{\Phi}}
&=
\nm \lambda{\ell ^\Phi}
\le
\sup _{\nm \mu{\ell ^{\Phi ^*}}\le 1}|(\lambda ,\mu)_{\ell ^2}|
\\[1ex]
&=
\sup _{S\in B_T}|(T,S)_{\mascI _2}|
\le
\sup _{S\in B}|(T,S)_{\mascI _2}|
=
\nmm T.
\end{aligned}
\end{equation}
This proves the result when $T$ is compact.

\par

If $\Phi$ is positive,
then $T$ is compact, and the result follows from
\eqref{Eq:EquivSchatNormRevEst}.

\par

Therefore assume that $\Phi$ is not positive, and let $\ep >0$.
By Proposition \ref{Prop:OrlSchattTop}, there are sequences
$\{ f_{k,j}\} _{j=1}^\infty \in \ON (\maclH _k)$, $k=1,2$, and such that
$$
\nm T{\mascI _\Phi}
\le
\nm {\{ (Tf_{1,j},f_{2,j})_{\maclH _2}\} _{j=1}^\infty}{\ell ^\Phi}+\frac \ep 2.
$$
Now we choose $S$ as in \eqref{Eq:DualExp}, and such that
$$
\nm {\{ (Tf_{1,j},f_{2,j})_{\maclH _2}\} _{j=1}^\infty}{\ell ^\Phi}
\le
\left |
\sum _{j=1}^\infty
(Tf_{1,j},f_{2,j})_{\maclH _2}\overline \mu _j
\right | +\frac \ep 2.
$$
This is possible in view of \eqref{Eq:LuxNormEquiv}. A combination
of these estimates yields
$$
\nm T{\mascI _\Phi}
\le
\left |
\sum _{j=1}^\infty
(Tf_{1,j},f_{2,j})_{\maclH _2}\overline \mu _j
\right | +\ep
\le
\nmm T+\ep.
$$
Since $\ep >0$ is arbitrary, we get the first inequality in
\eqref{Eq:AltNormSchatEst} in this case as well, and the result follows.
\end{proof}

\par

Proposition \ref{Prop:SymbTranslDil} below explains mapping
properties for Schatten class symbols and some
other symbol classes under dilations and modulations, i.{\,}e. under
mappings of the form
\begin{equation}
\label{Eq:SymbTranslDil}
a\mapsto a(\cdo -X)
\quad \text{and}\quad
a\mapsto e^{i\scal \cdo \Xi}a ,
\qquad
a\in \mascS '(\rr {2d}).
\end{equation}
Here $s_{A,0,1}(\rr {2d})=\{ 0\}$
and $s_{A,0,j}(\rr {2d})$ is the set of
all functions of the form
$$
\sum _{k=1}^{j-1} W_{f_{k},g_{k}},
\qquad
f_k,g_k\in L^2(\rr d),\ k=1,\dots ,j-1,
$$
when $j\ge 2$. We observe that $a\mapsto \op _A(a)$
is bijective from $s_{A,0,j}(\rr {2d})$ to $\mascI _{0,j}(L^2(\rr d))$,
in view of \eqref{Eq:AWigPseudoLink2}. We also set
\begin{equation}
\label{Eq:SymbRankSingValDef}
\widetilde \sigma _{A,j}(a)\equiv (2\pi )^{\frac d2}\sigma _j(\op _A(a)),
\end{equation}
when $A\in \GL (d,\mathbf R)$, $a\in \mascS '(\rr {2d})$
and $j\in \mathbf Z_+$. By \eqref{Eq:AWigPseudoLink2}
and Lemma \ref{Lemma:SymbOpMap} it follows that 
\begin{equation}
\widetilde \sigma _{A,j}(a) = \inf _{a_0\in s_{A,0,j}}\nm {a-a_0}{s_{A,\infty}},
\qquad
a\in \mascS '(\rr {2d}),\ j\in \mathbf Z_+.
\end{equation}

\par

\begin{prop}
\label{Prop:SymbTranslDil}
Let $\maclB \subseteq \ell _0'(\mathbf Z_+)$ be a quasi-Banach space, let
$A\in \GL (d,\mathbf R)$, and let $X,\Xi \in \rr {2d}$. Then the following is true:
\begin{enumerate}
\item 
the mappings in \eqref{Eq:SymbTranslDil}
restrict to isometric isomorphisms on $s_{A,\maclB}(\rr {2d})$.
In particular,
\begin{equation}
\label{Eq:SchattTranslDil}
\nm {e^{i\scal \cdo \Xi}a(\cdo -X)}{s_{A,\maclB}}
=
\nm a{s_{A,\maclB}},
\qquad
a\in s_{A,\maclB}(\rr {2d}),\ X,\Xi \in \rr {2d}\text ;
\end{equation}

\vrum

\item
the mappings in \eqref{Eq:SymbTranslDil}
restrict to a bijection on $s_{A,0,j}(\rr {2d})$ when $j\in \mathbf Z_+$
and to isomorphisms on $s_{A,0}(\rr {2d})$ and $s_{A,\flat}(\rr {2d})$;

\vrum

\item if $j\in \mathbf Z_+$, then
\begin{equation}
\widetilde \sigma _{A,j}(a(\cdo -X)) = \widetilde \sigma _{A,j}(e^{i\scal \cdo \Xi}a)
= \widetilde \sigma _{A,j}(a).
\end{equation}
\end{enumerate}
%
%
\end{prop}

\par

\begin{proof}
Let
$$
S_{0,Z}
\quad \text{and}\quad
S_{k,z},\quad k=1,\dots,  4,\ z\in \rr d,\ Z=(z,\zeta )\in \rr {2d},
$$
be the mappings on $\mascS '(\rr d)$, given by
\begin{equation}
\label{Eq:SymplCorrMaps}
\begin{alignedat}{2}
S_{0,Z}f &= f(\cdo -z)e^{i\scal \cdo \zeta}, &&
\\[1ex]
S_{1,\zeta}f(x) &= e^{i\scal {(I-A)x}\zeta}f(x), &
\quad
S_{2,\zeta}f (x) &= e^{-i\scal {Ax}\zeta}f(x),
\\[1ex]
S_{3,z}f(x) &= f(x+Az), &
\quad \text{and}\quad
S_{4,z}f(x) &= f(x+(A-I)z),
\end{alignedat}
\end{equation}
when $f\in \mascS '(\rr d)$.
By straight-forward computations, it follows that
\begin{equation}
\label{Eq:WignerTransl}
\begin{alignedat}{2}
W_{f,g}^A(X-Z)
&=
W_{S_{0,Z}f,S_{0,Z}g}^A(X), &
\quad
X,Z&\in \rr {2d},
\\[1ex]
e^{i\scal x\zeta}W_{f,g}^A(x,\xi )
&=
W_{S_{1,\zeta}f,S_{2,\zeta}g}^A(x,\xi ), &
\quad
x,\xi ,\zeta &\in \rr {d},
\\[1ex]
e^{i\scal z\xi}W_{f,g}^A(x,\xi )
&=
W_{S_{3,z}f,S_{4,z}g}^A(x,\xi ), &
\quad
x,z,\xi &\in \rr {d},
\end{alignedat}
\end{equation}
when $f,g\in \mascS '(\rr d)$. 

\par

It follows from \eqref{Eq:WignerTransl} that the mappings
\eqref{Eq:SymbTranslDil} restricts to bijections on
$s_{A,0,j}(\rr {2d})$. Since $s_{A,0}(\rr {2d})$ is the union of all
$s_{A,0,j}(\rr {2d})$, $j\ge 1$, it also follows that these
mappings are isomorphisms on $s_{A,0}(\rr {2d})$. The fact that
the mappings in \eqref{Eq:SymplCorrMaps} restrict to isomorphisms
on $\mascS (\rr d)$, shows that the mappings \eqref{Eq:SymbTranslDil} 
are also isomorphisms on $s_{A,\flat}(\rr {2d})$.

\par

It remains to show that the mappings \eqref{Eq:SymbTranslDil}
are bijective isometries on $s_{A,\maclB}(\rr {2d})$, and that (3) holds.
Let $a\in s_{A,\infty}(\rr {2d})$,
and let $\Omega$ be the set of all unit vectors in $L^2(\rr d)$.
Then it follows that the mappings in \eqref{Eq:SymplCorrMaps} are bijections
on $\Omega$. A combination of these facts and
\eqref{Eq:WignerTransl} gives
\begin{align*}
\nm {a(\cdo -Z)}{s_{A,\infty}}
&=
\sup _{f,g\in \Omega}| (a(\cdo -Z),W_{f,g}^A)_{L^2}|
=
\sup _{f,g\in \Omega}| (a,W_{S_{0,-Z}f,S_{0,-Z}g}^A)_{L^2}|
\\[1ex]
&=
\sup _{f,g\in \Omega}| (a,W_{f,g}^A)_{L^2}|
=
\nm a{s_{A,\infty}}.
\end{align*}
Hence translations are norm preserved on $s_{A,\infty}(\rr {2d})$. In similar ways
it also follows that the other mappings in \eqref{Eq:SymbTranslDil}
are norm preserved on $s_{A,\infty}(\rr {2d})$. In particular \eqref{Eq:SchattTranslDil}
holds when $\maclB =\ell ^\infty$, i.{\,}e.
\begin{equation}
\tag*{(\ref{Eq:SchattTranslDil})$'$}
\label{Eq:SchattTranslDilInfty}
\nm {e^{i\scal \cdo \Xi}a(\cdo -X)}{s_{A,\infty}}
=
\nm a{s_{A,\infty}},
\qquad
a\in s_{A,\infty}(\rr {2d}),\ X,\Xi \in \rr {2d}.
\end{equation}

\par

Now let $\maclL =\maclL _{X,\Xi}$ be the operator on $\mascI _\infty (L^2(\rr d))$, defined
by
$$
\maclL (\op _A(a)) = \op _A(e^{i\scal \cdo \Xi}a(\cdo -X)),
\qquad
a\in s_{A,\infty}(\rr {2d}).
$$
Then $\maclL$ is well-defined, because $a\mapsto \op _A(a)$
is an isometric isometry from $s_{A,\infty}(\rr {2d})$
to $\mascI _\infty (L^2(\rr d))$. A combination of the latter
isometry, (2), \ref{Eq:SchattTranslDilInfty}
and Proposition \ref{Prop:ContSchattInvariance} gives
\begin{alignat*}{2}
\sigma _j\big ( \maclL (\op _A(a))\big )
&=
\sigma _j(\op _A(a)),&
\qquad a&\in s_{A,\infty}(\rr {2d}),
\intertext{and}
\nm { \maclL (\op _A(a))}{\mascI _{\maclB}}
&=
\nm {(\op _A(a)}{\mascI _{\maclB}},&
\quad
a&\in s_{A,\maclB}(\rr {2d}).
\end{alignat*}
This shows that (1) and (3) hold, and the result follows.
\end{proof}

\par

Finally we have the following, which shows that $s_{\maclB} ^w(\rr {2d})$
is invariant under symplectic Fourier transform. Here the symplectic
Fourier transform of $a\in \mascS (\rr {2d})$ is defined by the formula
$$
(\mascF _\sigma a)(x,\xi ) = \pi ^{-d}\iint _{\rr {2d}}
a(y,\eta )e^{2i(\scal y\xi -\scal x\eta )}\, dyd\eta .
$$
The definition of $\mascF _\sigma$ extends in usual ways to a
homeomorphism on $\mascS '(\rr {2d})$ and to a unitary map
on $L^2(\rr {2d})$. Furthermore, $\mascF _\sigma ^2$ is the
identity operator. (See e.{\,}g. \cite{Fol}.)

\par

\begin{prop}
\label{Prop:SymplFourInvWeylCase}
Let $\maclB \subseteq \ell _0'(\mathbf Z_+)$ be a quasi-Banach space.
Then $\mascF _\sigma$ is an isometric isomorphism on $s_{\maclB} ^w(\rr {2d})$.
\end{prop}

\par

\begin{proof}
Let $K$ be the kernel of $\op ^w(a)$. Then it follows by straight-forward
computations that the kernel of $\op ^w(\mascF _\sigma a)$
is given by $K_1(x,y)=K(-x,y)$ (see e.{\,}g. \cite{Fol,Toft13}).
This implies that the singular values of $\op ^w(a)$
and $\op ^w(\mascF _\sigma a)$ agree, which gives the result.
\end{proof}

%
%
%
%
%

\par

For future references we observe that transitions
for convolutions and multiplications under symplectic
Fourier transformations are given by
\begin{equation}
\label{Eq:TransConvMultSymplFourT}
\mascF _\sigma (a*b) = \pi ^{-d} \mascF _\sigma a \cdot \mascF _\sigma b
\quad \text{and}\quad
\mascF _\sigma (a\cdot b) = \pi ^{d} \mascF _\sigma a * \mascF _\sigma b,
\end{equation}
when $a\in \mascS '(\rr {2d})$ and $b\in \mascS (\rr {2d})$.

\par

\section{Convolutions between Orlicz Schatten
symbol classes and Orlicz spaces}\label{sec3}

\par

In this section we deduce norm estimates for
convolutions
between elements in Orlicz Schatten-von Neumann
symbol spaces and Orlicz spaces.

\par

We begin with Theorem \ref{Thm:Conv1} below.
Here
the involved Young functions should satisfy
\begin{equation}
\label{Eq:YoungFuncCond}
\begin{aligned}
t_0t_1t_2
&\le
c_1\Phi _0^*(t_0)\Phi _1(t_1)+c_2\Phi _0^*(t_0)\Phi _2(t_2)
+c_0\Phi _1(t_1)\Phi _2(t_2),
\\[1ex]
t_0,t_1,t_2
&\in
[0,T],
\end{aligned}
\end{equation}
for some fixed $c_j,T>0$, $j=0,1,2$.

\par

\begin{example}
\label{Eq:ExYoungFuncCond}
Let $p_0,p_1,p_2\in [1,\infty ]$ be such that
the Young condition
\begin{equation}
\label{Eq:AltCondYoungFunc}
\frac 1{p_1}+\frac 1{p_2}=1+\frac 1{p_0}
\end{equation}
holds, and let $\Phi _j = \Phi _{[p_j]}$, $j=0,1,2$.
Then \eqref{Eq:YoungFuncCond} is fulfilled with
$$
c_0=\frac 1{p_0},
\quad
c_1=\frac 1{p_1'}
\quad \text{and}\quad
c_2=\frac 1{p_2'}. 
$$
\end{example}

\par

\begin{thm}
\label{Thm:Conv1}
Let $\Phi _j$, $j=0,1,2$, be Young functions such that
\eqref{Eq:YoungFuncCond} holds for some
constants $c_j,T>0$, $j=0,1,2$.
Also let $A_1,A_2\in \GL (d,\mathbf R)$ be such that
$A_1+A_2=I_d$.
Then the map $(a_1,a_2)\mapsto a_1*a_2$ from
$\mascS (\rr {2d})\times \mascS (\rr {2d})$
to
$\mascS (\rr {2d})$ extends to a continuous map
from $s_{A_1,\Phi _1}(\rr {2d})\times s_{A_2,\Phi _2}(\rr {2d})$
to $L^{\Phi _0}(\rr {2d})$, and
\begin{equation}
\label{Eq:ConvEst3A}
\begin{aligned}
\nm {a_1*a_2}{L^{\Phi _0}}
&\le
2((2\pi )^dc_0+c_1+c_2 )\nm {a_1}{s_{A_1,\Phi _1}}\nm {a_2}{s_{A_2,\Phi _2}},
\\[1ex]
a_1&\in s_{A_1,\Phi _1}(\rr {2d}),\ a_2\in s_{A_2,\Phi _2}(\rr {2d}).
\end{aligned}
\end{equation}
If in addition $\Phi _1$ or $\Phi _2$ satisfies a local $\Delta _2$-condition,
then the extension is unique.
\end{thm}

\par

\begin{thm}
\label{Thm:Conv2}
Let $\Phi _j$, $j=0,1,2$, be Young functions such that
\eqref{Eq:YoungFuncCond} holds for some
constants $c_j,T>0$, $j=0,1,2$.
Also let $A\in \GL (d,\mathbf R)$.
Then the map $(a_1,a_2)\mapsto a_1*a_2$ from
$\mascS (\rr {2d})\times \mascS (\rr {2d})$ to
$\mascS (\rr {2d})$ extends to a continuous map
from $s_{A,\Phi _1}(\rr {2d})\times L^{\Phi _2}(\rr {2d})$
to $s_{A,\Phi _0}(\rr {2d})$, and 
\begin{equation}
\label{Eq:ConvEst4}
\begin{aligned}
\nm {a_1*a_2}{s_{A,\Phi _0}}
&\le
2(c_0+c_1+(2\pi )^dc_2)\nm {a_1}{s_{A,\Phi _1}}\nm {a_2}{L^{\Phi _2}},
\\[1ex]
a_1&\in s_{A,\Phi _1}(\rr {2d}),\ a_2\in L^{\Phi _2}(\rr {2d}).
\end{aligned}
\end{equation}
If in addition $\Phi _1$ or $\Phi _2$ satisfies a local $\Delta _2$-condition,
then the extension is unique.
\end{thm}

\par

\begin{rem}
We observe that if $\Phi _j$ are the same as in 
Example \ref{Eq:ExYoungFuncCond}, then 
Theorems \ref{Thm:Conv1} and \ref{Thm:Conv2} show
that the convolution estimates hold true for classical
Schatten-von Neumann classes and Lebesgue spaces
with Lebesgue exponents satisfying the ordinary Young
condition \eqref{Eq:AltCondYoungFunc}. Hence
Theorems \ref{Thm:Conv1} and \ref{Thm:Conv2}
extend the classical convolution results in \cite{Toft3,Wer}.
\end{rem}

\par

For the proofs of Theorems \ref{Thm:Conv1} and \ref{Thm:Conv2}
we need the following lemma.

\par

\begin{lemma}
\label{Lemma:ConvSchattExp}
Let $A=\frac 12I_d$, $a$ be as in \eqref{Eq:SpectExp0Symb} and
let $b\in s_{A,1}(\rr {2d})$ be given by
\begin{equation}
\label{Eq:SymbbExp}
b=\sum _{j=1}^\infty \mu _jW_{g_{2,j},g_{1,j}},
\end{equation}
for some non-negative decreasing sequence
$\mu =\{ \mu _j\} _{j=1}^\infty \in \ell ^\sharp (\mathbf Z_+)$,
$\{ g_{n,j} \} _{j=1}^\infty \in \ON (L^2(\rr d))$, $n=1,2$. Then
\begin{equation}
\label{Eq:ConvExp}
|a*b(X)| \le \sum _{j,k=1}^\infty \lambda _j\mu _k h_{j,k}(X),
\end{equation}
for some sequence $\{h_{j,k}\} _{j,k=1}^\infty$ of non-negative
functions in $L^1(\rr {2d})\cap L^\infty (\rr {2d})$ such that
\begin{equation}
\label{Eq:ConvSchattExp}
\begin{aligned}
\sup _{k\ge 1}\left (\sum _{j=1}^\infty h_{j,k}(X)\right )&\le 1,
\qquad
\sup _{j\ge 1}\left (\sum _{k=1}^\infty h_{j,k}(X)\right )\le 1
\\[1ex]
\text{and}\quad
\nm {h_{j,k}}{L^1}&\le (2\pi )^{d},\quad j,k\ge 1,
\end{aligned}
\end{equation}
are fulfilled.
\end{lemma}

\par

\begin{proof}
The estimate \eqref{Eq:ConvExp} holds with
$$
h_{j,k}(X)=|W_{f_{2,j},f_{1,j}} * W_{g_{2,k},g_{1,k}}(X)|. 
$$
By straight-forward application of Fourier's inversion formula we obtain
$$
h_{j,k} = \left ( \frac \pi 2\right )^d |u_{j,k}\cdot v_{j,k}| ,
$$
where
$$
u_{j,k} (x,\xi ) = W_{f_{2,j},g_{1,k}}(\textstyle{\frac x2},-\textstyle{\frac \xi 2})
\quad \text{and}\quad
v_{j,k} (x,\xi ) = W_{f_{1,j},g_{2,k}}(\textstyle{\frac x2},-\textstyle{\frac \xi 2}).
$$

\par

Since $f_{n,j}$ and $g_{n,k}$ are unit vectors, $n=1,2$, it follows from Moyal's formula
that
$$
\nm {W_{f_{2,j},g_{1,k}}}{L^2}=\nm {W_{f_{1,j},g_{2,k}}}{L^2}=1.
$$
This implies that
$$
\nm {u_{j,k}}{L^2}=\nm {v_{j,k}}{L^2}=2^d,
$$
by a straight-forward change of variables of integrations. By
Cauchy-Schwarz inequality we get
$$
\nm {h_{j,k}}{L^1}
=
\left ( \frac \pi 2\right )^d \nm {u_{j,k} \cdot v_{j,k}}{L^1}
\le
\left ( \frac \pi 2\right )^d \nm {u_{j,k}}{L^2}\nm {v_{j,k}}{L^2}
= (2\pi )^d,
$$
which gives the last inequality in \eqref{Eq:ConvSchattExp}.

\par

Since $\{ f_{n,k}\} _{k=1}^\infty$ and $\{ g_{n,k}\} _{k=1}^\infty$
are orthonormal sequences in
$L^2(\rr d)$, it follows by straight-forward computations that if
\begin{alignat*}{2}
f_{n,k,x}(y)
&=
2^{-\frac d2}f_{n,k}(x+\textstyle{\frac y2}),&
\qquad
x,y &\in \rr d,
\intertext{and}
g_{n,k,x,\xi}(y)
&=2^{-\frac d2}g_{n,k}(x+\textstyle{\frac y2})e^{i\scal y\xi},&
\qquad
x,y,\xi &\in \rr d,
\end{alignat*}
then $\{ f_{n,k,x}\} _{k=1}^\infty$ and $\{ g_{n,k,x,\xi}\} _{k=1}^\infty$ are
also orthonormal sequences in
$L^2(\rr d)$, $n=1,2$, for every $x,\xi \in \rr d$. Hence for
$$
u_{j,k}(x,\xi ) = 2^d(2\pi)^{-\frac d2}(f_{2,j,x},g_{1,k,x,\xi })_{L^2}
$$
and
$$
v_{j,k}(x,\xi ) = \left (\frac 2\pi \right )^{\frac d2}(f_{1,j,x},g_{2,k,x,\xi })_{L^2},
$$
Parseval's formula gives
$$
\sum _{k=1}^\infty
|u_{j,k}|^2
=
\left (\frac 2\pi \right )^d
\sum _{k=1}^\infty
|(f_{2,j,x},g_{1,k,x,\xi })_{L^2}|^2
\le
\left (\frac 2\pi \right )^d\nm {f_{2,j,x}}{L^2}^2
=
\left (\frac 2\pi \right )^d.
$$
In the same way we reach
$$
\sum _{k=1}^\infty |v_{j,k}|^2
\le
\left (\frac 2\pi \right )^d.
$$

\par

An application of Cauchy-Schwarz inequality now gives
\begin{align*}
\sum _{k=1}^\infty h_{j,k}(x,\xi )
&=
\left ( \frac \pi 2\right )^d
\sum _{k=1}^\infty
|u_{j,k}\cdot v_{j,k}|
\\[1ex]
&\le
\left ( \frac \pi 2\right )^d
\left (
\sum _{k=1}^\infty |u_{j,k}|^2
\right )^{\frac 12}
\left (
\sum _{k=1}^\infty |v_{j,k}|^2
\right )^{\frac 12}
\\[1ex]
&\le
\left ( \frac \pi 2\right )^d
\left ( \frac 2\pi \right )^d =1,
\end{align*}
which gives the second inequality in 
\eqref{Eq:ConvSchattExp}. The first
inequality follows by similar arguments, after
swapping the roles of $f_{n,j}$ and $g_{n,j}$.
The details are left for the reader.
\end{proof}

\par

\begin{proof}[Proof of Theorem \ref{Thm:Conv1}]
By Proposition \ref{Prop:DiscreteNorms2},
Remark \ref{Rem:SmallSchattDenseSchwartz}, and straight-forward
applications of Hahn-Banach's theorem, the result follows if
we prove that \eqref{Eq:ConvEst3A} holds for $a_1\in s_{A_1,\flat}(\rr {2d})$
and $a_2\in s_{A_2,\flat}(\rr {2d})$ such that $\nm {a_m}{s_{A_m,\Phi _m}}=1$, $m=1,2$.
Moreover, let
$$
B_m=A_m-\frac 12\cdot I_d,
\qquad
m=1,2,
$$
and note that if
$$
a=e^{i\scal {B_1D_\xi}{D_x}}a_1
\quad \text{and}\quad
b=e^{i\scal {B_2D_\xi}{D_x}}a_2,
$$
then, since $B_1+B_2=0$, due to the assumptions, we get
$$
a_1*a_2 =(e^{i\scal {B_1D_\xi}{D_x}}a_1)*(e^{i\scal {B_2D_\xi}{D_x}}a_2)
=a*b
$$
provided the convolutions are well-defined. Since $a,b\in s_\flat ^w(\rr {2d})$,
$$
\nm a{s_\Phi ^w}=\nm {a_1}{s_{A_1,\Phi}}
\quad \text{and}\quad
\nm b{s_\Phi ^w}=\nm {a_2}{s_{A_2,\Phi}},
$$
in view of Remark \ref{Rem:SchattCalculTransf}, we
reduce ourselves to the Weyl case when $A_1=A_2=\frac 12I_d$.


\par

Choose expansions for $a$ and $b$ as in Lemma
\ref{Lemma:ConvSchattExp}. Then it follows from
the assumptions and Lemma \ref{Lemma:ConvSchattExp}
that for any $u\in L^{\Phi _0^*}(\rr {2d})$ with $\nm u{L^{\Phi _0^*}}= 1$,
we have
$$
|(a*b,u)_{L^2}|\le c_0J_0+c_1J_1+c_2J_2,
$$
where
\begin{align*}
J_0
&=
\sum _{j,k=1}^\infty \int _{\rr {2d}}\Phi _1(\lambda _j)\Phi _2(\mu _k)h_{j,k}(X)\, dX,
\\[1ex]
J_1
&=
\sum _{j,k=1}^\infty \int _{\rr {2d}}\Phi _0^*(|u(X)|)\Phi _1(\lambda _j)h_{j,k}(X)\, dX
\intertext{and}
J_2
&=
\sum _{j,k=1}^\infty \int _{\rr {2d}}\Phi _0^*(|u(X)|)\Phi _2(\mu _k)h_{j,k}(X)\, dX.
\end{align*}
For $J_1$ we have
\begin{align*}
J_1
&=
\sum _{j,k=1}^\infty \int _{\rr {2d}}\Phi _0^*(|u(X)|)\Phi _1(\lambda _j)h_{j,k}(X)\, dX
\\[1ex]
&\le
\sum _{j=1}^\infty \int _{\rr {2d}}\Phi _0^*(|u(X)|)\Phi _1(\lambda _j)\, dX
\\[1ex]
&\le
\left ( \sum _{j=1}^\infty \Phi _1(\lambda _j) \right )
\left (\int _{\rr {2d}}\Phi _0^*(|u(X)|)\, dX\right )
\le 1.
\end{align*}
Here the first inequality follows from \eqref{Eq:ConvSchattExp}
and the third one follows from the fact that
$\nm a{s_{\Phi _1}^w}=\nm u{L^{\Phi _0^*}}=1$. Hence
$J_1\le 1$. In the same way we get $J_2\le 1$.

\par

For $J_0$ we have
\begin{align*}
J_0
&=
\sum _{j,k=1}^\infty \int _{\rr {2d}}\Phi _1(\lambda _j)\Phi _2(\mu _k)h_{j,k}(X)\, dX
\\[1ex]
&=
\sum _{j,k=1}^\infty \Phi _1(\lambda _j)\Phi _2(\mu _k )\nm{h_{j,k}}{L^1}
\\[1ex]
&\le
(2\pi )^d
\left (\sum _{j=1}^\infty \Phi _1(\lambda _j) \right )
\left (\sum _{k=1}^\infty\Phi _2(\mu _k )\right )
\le (2\pi )^d.
\end{align*}
Here the first inequality follows from \eqref{Eq:ConvSchattExp}
and the second one follows from the fact that
$\nm a{s_{\Phi _1}^w}=\nm b{s_{\Phi _2}^w}=1$. Hence
$J_0\le  (2\pi )^d$.

\par

A combination of these inequalities shows that
$$
|(a*b,u)_{L^2}|\le (2\pi )^dc_0+c_1+c_2
$$
holds when $\nm u{L^{\Phi _0^*}}=\nm a{s_{\Phi _1}^w}=\nm b{s_{\Phi _2}^w}=1$.
By taking the supremum over all such $u$, and using
Proposition \ref{Prop:SchattSymbEquivNorm}, we arrive at
\begin{equation}
\label{Eq:ConvEst3Prel}
\nm{a*b}{L^{\Phi _0}}\le 2((2\pi )^dc_0+c_1+c_2)
\end{equation}
for such $a$ and $b$. The estimate \eqref{Eq:ConvEst3A}
now follows from \eqref{Eq:ConvEst3Prel} and homogeneity, and the result follows.
%
%
%
%
%
%
%
\end{proof}

\par

\begin{proof}[Proof of Theorem \ref{Thm:Conv2}]
In the same way as in the proof of Theorem \ref{Thm:Conv1},
we reduce ourselves to show that \eqref{Eq:ConvEst4}
should hold for $A=\frac 12\cdot I_d$, $a=a_1\in s_\flat ^w(\rr {2d})$,
and $b=a_2\in \mascS (\rr {2d})$.
Let $u \in s_{\Phi _0^*} ^w(\rr {2d})$, and first suppose that
$$
\nm a{s_{\Phi _1}^w}=\nm b{L^{\Phi _2}}=\nm u{s_{\Phi _0^*}^w}=1.
$$
Then Theorem \ref{Thm:Conv1} and its proof give
\begin{align*}
|(a*b,u)|=|(\check a*\overline u,\overline b)|
&\le
(c_0+c_1+(2\pi )^dc_2)
\nm {\check a}{s_{\Phi _1}^w}
\nm {\overline u}{s_{\Phi _0^*}^w}
\nm {\overline b}{L^{\Phi _2}}
\\[1ex]
&=
c_0+c_1+(2\pi )^dc_2.
\end{align*}
The assertion now follows from duality, using (1) and (3) in
Proposition \ref{Prop:OrlDuality}, and Proposition
\ref{Prop:SchattSymbEquivNorm}. The details are left for the reader.
\end{proof}

\par

\begin{cor}\label{Cor:OrlSchConvL1}
Let $\Phi$ be a Young function
which satisfies a local $\Delta _2$-condition, and let
$A,B\in \GL (d,\mathbf R)$ be such that
$$
A+B=I.
$$
\begin{enumerate}
\item The map $(a,b)\mapsto a*b$ is continuous from
$s_{A,\Phi}(\rr {2d})\times L^1(\rr {2d})$ to $s_{A,\Phi}(\rr {2d})$.

\vrum

\item The map $(a,b)\mapsto a*b$ is continuous from
$s_{A,\Phi}(\rr {2d})\times s_{B,\Phi ^*}(\rr {2d})$ to $L^\infty (\rr {2d})$.

\vrum

\item The map $(a,b)\mapsto a*b$ is continuous from
$s_{A,\Phi}(\rr {2d})\times s_{B,1}(\rr {2d})$
to $L^{\Phi}(\rr {2d})$.

\vrum

\item The map $(a,b)\mapsto a*b$ is continuous from
$s_{A,\Phi}(\rr {2d})\times L^{\Phi ^*}(\rr {2d})$ to $s_{A,\infty} (\rr {2d})$.
\end{enumerate}
\end{cor}

\par

\begin{proof}
The assertions (1) and (3) follow by choosing
$$
\Phi _0=\Phi _1=\Phi
\quad \text{and}\quad
\Phi _2(t)=t
$$
in Theorems \ref{Thm:Conv1} and \ref{Thm:Conv2}.
The other assertions follow by choosing
$$
\Phi _0=\Phi _{[\infty ]},\quad
\Phi _1=\Phi
\quad \text{and}\quad
\Phi _2=\Phi ^*
$$
in Theorems \ref{Thm:Conv1} and \ref{Thm:Conv2}.
\end{proof}

%

\par

\begin{rem}
Suppose that the Young functions
$\Phi _0$, $\Phi _1$ and $\Phi _2$ satisfy
\begin{equation}
\label{Eq:YoungCond1}
\Phi _1^{-1}(s)\Phi _2^{-1}(s)
\le
Cs\Phi _0^{-1}(s), \quad s\ge 0,
\end{equation}
for some $C>0$. Then the conclusions in
Theorems \ref{Thm:Conv1} and \ref{Thm:Conv2}
hold true with $c_0=c_1=c_2=C$.

\par

In fact, if \eqref{Eq:YoungCond1} holds, then 
the arguments in the proof of Lemma 6 in \cite[Section 3.3]{RaoRen1}
give
\begin{equation}
\label{Eq:YoungCond2}
\begin{aligned}
t_0t_1t_2
&\le
C\big (
\Phi _0^*(t_0)\Phi _1(t_1)+\Phi _0^*(t_0)\Phi _2(t_2)+\Phi _1(t_1)\Phi _2(t_2)
\big ),
\\[1ex]
t_0,t_1,t_2 &\ge 0,
\end{aligned}
\end{equation}
and the assertion follows. For the reader's convenience,
a proof of that \eqref{Eq:YoungCond1} implies
\eqref{Eq:YoungCond2}, in a more general situation,
is included in Appendix \ref{App:B}. (See Proposition
\ref{Prop:YoungCondEsts}.)
\end{rem}

\par

As a first application of the previous convolution
results we have the following concerning compactly
supported elements in $s_{A,\Phi}(\rr {2d})$.
Here $\mascF a=\widehat a$ denotes a Fourier transform of $a\in \mascS '(\rr {2d})$.
The statement is independent of the choice of the Fourier transform, and for that
reason the choice is not specified. Here recall from \cite{Ho1} that the set of
all compactly supported distributions on $\rr d$ is denoted by $\mascE '(\rr d)$.
%
%

\par

\begin{prop}
Let $\Phi$ be a Young function and $A\in \GL (d,\mathbf R)$.
Then the following is true:
\begin{enumerate}
\item $s_{A,\Phi}(\rr {2d})\cap \mascF \mascE '(\rr {2d})
=
L^\Phi (\rr {2d})\cap \mascF \mascE '(\rr {2d})$;

\vrum

\item $s_{\Phi}^w(\rr {2d})\cap \mascE '(\rr {2d})
=
\mascF L^\Phi (\rr {2d})\cap \mascE '(\rr {2d})$.
\end{enumerate}
\end{prop}

\par

\begin{proof}
Suppose that $a\in \mascF \mascE '(\rr {2d})$, and let
$b\in \mascS (\rr {2d})$ be such that $\widehat b =1$
on $\supp \widehat a$, when the Fourier transform
is defined by
$$
(\mascF f)(\xi ) = \widehat f(\xi )=(2\pi)^{-\frac d2}\int _{\rr d}f(x)e^{-i\scal x\xi}\, dx,
\qquad f\in \mascS (\rr d).
$$
Then
$$
a=(2\pi )^{-\frac d2}a*b.
$$
By Corollary \ref{Cor:OrlSchConvL1} we obtain
%
%
$$
\nm a{s_{A,\Phi}}
=
(2\pi )^{-\frac d2}
\nm {a*b}{s_{A,\Phi}}
\lesssim
\nm a{L^\Phi}\nm b{s_{A,1}}\asymp \nm a{L^\Phi},
$$
and
$$
\nm a{L^\Phi}
=
(2\pi )^{-\frac d2}
\nm {a*b}{L^\Phi}
\lesssim
\nm a{s_{A,\Phi}}\nm b{s_{B,1}}
\asymp
\nm a{s_{A,\Phi}}.
$$
A combination of these estimates gives (1).

\par

The assertion (2) is invariant under the choice of Fourier transform.
The result then follows by combining (1) with Proposition
\ref{Prop:SymplFourInvWeylCase}, taken into account that the convolutions
are uniquely defined, since $\mascS (\rr {2d})$ is dense in $s_{A,1}(\rr {2d})$
and in $s_{B,1}(\rr {2d})$.
\end{proof}

\par

\par

\section{Dilated multiplications and
convolutions for Orlicz Schatten-von
Neumann classes}\label{sec4}

\par

In this section we deduce norm estimates for
dilated convolutions
between elements in Orlicz Schatten-von Neumann
symbols. Especially we show that such estimates
hold true when the involved Young functions satisfy
some sort of Young condition given in
\eqref{Eq:YoungFuncCond}.

\par

More precisely, we have the following result.

\par

\begin{thm}
\label{Thm:OrlSchConvL2}
Let $\Phi _j$ be Young functions
such that \eqref{Eq:YoungFuncCond} holds for some
$c_j,T>0$, $j=0,1,2$.
Also suppose that $t_1,t_2\in \mathbf R\setminus 0$ satisfy
$$
m_1{t_1^{-2}}+m_2 {t_2^{-2}}=1
$$
for some choices of $m_j\in \{ -1,1\}$, $j=1,2$.
Then the map $$(a_1,a_2)\mapsto a_1(t_1\cdo )*a_2(t_2\cdo )$$
from $\mascS (\rr {2d})\times \mascS (\rr {2d})$ to
$\mascS (\rr {2d})$ extends to a continuous map from
$s_{\Phi _1} ^w(\rr {2d})\times s_{\Phi _2} ^w(\rr {2d})$ to
$s_{\Phi _0} ^w(\rr {2d})$,
and
\begin{equation}
\label{Eq:OrlSchConvL1Eq1Source}
\begin{aligned}
\nm {a_1(t_1\cdo )&*a_2(t_2\cdo )}{s_{\Phi _0} ^w}
\\
&\le 2(2\pi )^{\frac d2}\left (c_0+\sum _{j=1}^2c_j |t_j|^{2d} \right )
\prod _{j=1}^2 \left ( |t_j|^{-2d}\nm {a_j}{s_{\Phi _j}^w}\right ),
\\[1ex]
a_1&\in s_{\Phi _1} ^w(\rr {2d}), a_2\in s_{\Phi _2}^w(\rr {2d}).
\end{aligned}
\end{equation}

\par

If in addition $\Phi _1$ or $\Phi _2$ satisfy a local $\Delta _2$-condition,
then $a_1(t_1\cdo )*a_2(t_2\cdo )$ in {\rm{\ref{Eq:OrlSchConvL1Eq1}}}
is uniquely defined.

\par

Moreover, if $\op ^w(a_j)$, $j=1,2$, are positive semi-definite
operators, then the same holds true for $\op ^w(a_1(t_1\cdo )*a_2(t_2\cdo ))$. 
\end{thm}

\par

We observe that \eqref{Eq:OrlSchConvL1Eq1Source} can
also be expressed as
\begin{equation}
\tag*{(\ref{Eq:OrlSchConvL1Eq1Source})$'$}
\label{Eq:OrlSchConvL1Eq1}
\begin{aligned}
\nm {a_1&(t_1\cdo )*a_2(t_2\cdo )}{s_{\Phi _0} ^w}
\\
&\le 2(2\pi )^{\frac d2}\cdot (c_0(t_1t_2)^{-2d}+c_1t_1^{-2d}+c_2t_2^{-2d})
\nm {a_1}{s_{\Phi _1} ^w}\nm {a_2}{s_{\Phi _2}^w},
\\[1ex]
a_1&\in s_{\Phi _1} ^w(\rr {2d}),\ a_2\in s_{\Phi _2}^w(\rr {2d}).
\end{aligned}
\end{equation}

\par

There is a natural extension of Theorem \ref{Thm:OrlSchConvL2}
to the multi-linear case, which is given in the next theorem. Here
the Young functions should fulfill 
\begin{equation}
\label{Eq:YoungFuncCondGen}
\begin{aligned}
t_0\cdots t_{N}
&\le 
\left (
\sum _{j=1}^N
c_j\prod _{m\in \Omega _j}
\Phi _m(t_m)
\right )
\Phi _0^*(t_0)
+c_0\prod _{j=1}^N\Phi _j(t_j),
\\[1ex]
\Omega _j
&=
\{1,\dots ,N\} \setminus \{ j\},
\qquad
t_0,\dots ,t_N\in [0,T],
\end{aligned}
\end{equation}
for some constants $c_j,T>0$, $j=0,\dots ,N$.

\par

\begin{thm}\label{Thm:OrlSchConvL2MultLin}
Let $\Phi _j$ be Young functions
such that {\rm{\eqref{Eq:YoungFuncCondGen}}} holds for some
constants $c_j,T>0$, $j=0,\dots ,N$.
Also suppose that $t_1,\dots ,t_N\in \mathbf R\setminus 0$ satisfy
\begin{equation}
\label{Eq:tCondMultilinCase}
\sum _{j=1}^N  m_jt_j^{-2} = 1,
\end{equation}
for some choices of $m_j\in \{ -1,1\}$, $j=1,\dots ,N$.
Then the map
$$
(a_1,\dots ,a_N)\mapsto a_1(t_1\cdo )*\cdots *a_N(t_N\cdo )
$$
from $\mascS (\rr {2d})\times \cdots \times \mascS (\rr {2d})$ to
$\mascS (\rr {2d})$ extends to a continuous map from
$s_{\Phi _1} ^w(\rr {2d})\times \cdots \times s_{\Phi _N} ^w(\rr {2d})$ to
$s_{\Phi _0} ^w(\rr {2d})$,
and
\begin{equation}
\tag*{(\ref{Eq:OrlSchConvL1Eq1Source})$''$}
\label{Eq:OrlSchConvL1Eq1Mult}
\begin{aligned}
\nm {a_1(t_1\cdo )&*\cdots *a_N(t_N\cdo )}{s_{\Phi _0} ^w}
\\
&\le 2(2\pi )^{(N-1)\frac d2}\left (c_0+\sum _{j=1}^Nc_j|t_j|^{2d} \right )
\prod _{j=1}^N \left ( |t_j|^{-2d}\nm {a_j}{s_{\Phi _j}^w}\right ),
\\[1ex]
a_1&\in s_{\Phi _1} ^w(\rr {2d}),\dots , a_N\in s_{\Phi _N}^w(\rr {2d}).
\end{aligned}
\end{equation}

\par

If in addition at least $N-1$ of $\Phi _1,\dots ,\Phi _N$
satisfy a local $\Delta _2$-condition, then
$a_1(t_1\cdo )*\cdots *a_N(t_N\cdo )$ in
{\rm{\ref{Eq:OrlSchConvL1Eq1Mult}}}
is uniquely defined.

\par

Moreover, if $\op ^w(a_j)$, $j=1,\dots ,N$, are positive semi-definite
operators, then the same holds true for
$\op ^w(a_1(t_1\cdo )*\cdots *a_N(t_N\cdo ))$. 
\end{thm}

\par

In what follows we give a proof of Theorem \ref{Thm:OrlSchConvL2},
while the proof of the more general Theorem \ref{Thm:OrlSchConvL2MultLin}
is given in Appendix \ref{App:A}.

\par

For the proof of Theorem \ref{Thm:OrlSchConvL2}
we need the following lemma, which is equivalent to
\cite[Lemma 3.2]{Toft3}. The proof is therefore omitted. Here
$S_{0,z}$ and $S_{1,z}$ are mappings on $\rr {2d}$,
and $\maclK$ is the map on $\mascS '(\rr {2d})$, given by
\begin{equation}
\label{Eq:OpsDilConv}
\begin{aligned}
S_{0,z}(x,y) &= (z+x,z+y),
\quad
S_{1,z}(x,y) = (z-x,z-y),
\\[1ex]
(\maclK a)(x,y) &= (2\pi )^{-\frac d2}
\big ((I\otimes \mascF ^{-1})a\big ) ({\textstyle{\frac 12}}(x+y),x-y),
\\[1ex]
a&\in \mascS '(\rr {2d})\ \ x,y,z\in \rr d.
\end{aligned}
\end{equation}
We observe that 
$$
(\maclK a)(x,y) = (2\pi )^{-d}\int _{\rr d}
a({\textstyle{\frac 12}}(x+y),\xi )e^{i\scal {x-y}\xi}\, d\xi,
$$
when $a\in \mascS (\rr {2d})$, and that $\maclK a$
is the operator kernel
of $\op ^w(a)$ (cf. \eqref{atkernel}).

\par

\begin{lemma}
\label{Lemma:DilConvIdent}
Suppose that $s,t\in \mathbf R$ satisfy $(-1)^js^{-2}+(-1)^kt^{-2}=1$,
for some choice of $j,k\in \{ 0,1\}$, and let $S_{0,z}$, $S_{1,z}$ and
$\maclK$ be as in \eqref{Eq:OpsDilConv}.
Then
\begin{multline*}
\maclK (a(s\cdo )*b(t\cdo ))
=
(2\pi )^{d}|st|^{-d}
\int _{\rr d} (\maclK a) (S_{j,z/s} \cdo ) (\maclK b)(S_{k,-z/t} \cdo )\, dz,
\\[1ex]
a,b\in \mascS (\rr {2d}).
\end{multline*}
\end{lemma}

\par

\begin{proof}[Proof of Theorem \ref{Thm:OrlSchConvL2}]
We use the same notation as in Lemma \ref{Lemma:DilConvIdent},
and prove the result only for $j=0$ and $k=1$. The other
cases follow by similar arguments and are left for the reader.

\par

As for the proof of Theorem \ref{Thm:Conv1},
we reduce ourselves to show that \ref{Eq:OrlSchConvL1Eq1}
holds true when $a=a_1\in s_\flat ^w (\rr {2d})$, $b=a_2\in s_\flat (\rr {2d})$,
by a suitable combination of Proposition \ref{Prop:DiscreteNorms2},
Remark \ref{Rem:SmallSchattDenseSchwartz}, and Hahn-Banach's theorem.

\par

Assume that $u\in s_{\Phi _0^*}^w(\rr {2d})$, and suppose that
$\nm a{s_{\Phi _1}^w}=\nm b{s_{\Phi _2}^w}=\nm u{s_{\Phi _0^*}^w}=1$.
Then
$$
a=\sum _{j=1}^\infty\lambda _jW_{f_{1,j},f_{2,j}},
\quad
b=\sum _{k=1}^\infty \mu _kW_{g_{1,k},g_{2,k}}
\quad \text{and}\quad
u=\sum _{m=1}^\infty \nu _mW_{h_{1,m},h_{2,m}},
$$
for some sequences
$$
\lambda =\{\lambda _j\} _{j=1}^\infty,
\quad
\mu =\{\mu _k\} _{k=1}^\infty
\quad \text{and}\quad
\nu =\{\nu _m \} _{m=1}^\infty
$$
of non-negative real numbers, and
$$
\{ f_{n,j}\} _{j=1}^\infty ,\ \{ g_{n,k}\} _{k=1}^\infty ,\ \{ h_{n,m}\} _{m=1}^\infty 
\in \ON (L^2(\rr {d})),\quad n=1,2.
$$
Since $a,b,u\in \mascS (\rr {2d})$, we have
\begin{align*}
\sup _{j\ge 1}|\lambda _jj^N| <\infty ,
\quad
\sup _{k\ge 1}|\mu _kk^N| &<\infty ,
\quad
\sup _{m\ge 1}|\nu _mm^N|<\infty ,
\\[1ex]
\sup _{j\ge 1}\left (
\sup _{|\alpha +\beta |\le N}
\nm {x^\alpha \partial ^\beta f_{n,j}}{L^\infty}
\right ) &<\infty ,
\\[1ex]
\sup _{k\ge 1}\left (
\sup _{|\alpha +\beta |\le N}
\nm {x^\alpha \partial ^\beta g_{n,k}}{L^\infty}
\right ) &<\infty 
\intertext{and}
\sup _{m\ge 1}\left (
\sup _{|\alpha +\beta |\le N}
\nm {x^\alpha \partial ^\beta h_{n,m}}{L^\infty}
\right ) &<\infty ,
\qquad N\ge 1,\ n=1,2,
\end{align*}
in view of Proposition \ref{Prop:SchwKerOpProp}.
Hence, the expansions for $a$, $b$ and $u$ possess several strong convergence
properties.

\par

Let $s=t_1$ and $t=t_2$.
By Lemma \ref{Lemma:DilConvIdent}, and the fact that $(2\pi )^{\frac d2}\maclK$
is unitary on $L^2(\rr {2d})$ we obtain
\begin{equation}
\label{Eq:ConvExpFormulation}
\begin{aligned}
|(a(s\cdo )*b(t\cdo ),u)_{L^2}|
&=
(2\pi )^d
|(\maclK (a(s\cdo )*b(t\cdo )),\maclK u)_{L^2}|
\\[1ex]
&=
(2\pi )^{\frac d2}
|st|^{-d}\left |
\sum _{j,k,m=1}^\infty \lambda _j\mu _k\nu _m
H(j,k,m)
\right |,
\end{aligned}
\end{equation}
where
\begin{align*}
&H(j,k,m)
\\[1ex]
&=
(2\pi )^{\frac {3d}2}
\int _{\rr {d}}
\big (
(\maclK W_{f_{1,j},f_{2,j}}) (S_{0,z/s}\cdo )
(\maclK W_{g_{1,k},g_{2,k}})(S_{1,-z/t}\cdo ),
\maclK W_{h_{1,m},h_{2,m}}\big )_{L^2}\, dz
\\[1ex]
&=
\int _{\rr {d}}
\big (
(f_{1,j}\otimes \overline{f_{2,j}})(S_{0,z/s}\cdo )
(g_{1,k} \otimes \overline{g_{2,k}})(S_{1,-z/t}\cdo ),
h_{1,m}\otimes \overline{h_{2,m}}\, \big )_{L^2}\, dz 
\\[1ex]
&=
\int _{\rr {d}}
(F_{1,j,k}(\cdo ,z),h_{1,m})_{L^2(\rr d)}(h_{2,m},F_{2,j,k}(\cdo ,z))_{L^2(\rr d)}\, dz.
\end{align*}
Here
\begin{align*}
F_{1,j,k}(x,z) &= f_{1,j}(sx+z/s)\overline{g_{2,k}(-tx-z/t)}
\intertext{and}
F_{2,j,k}(y,z) &= \overline{f_{2,j}(sy+z/s)}g_{1,k}(-ty-z/t).
\end{align*}

\par

Next we prove the first and third inequality in
\begin{equation}
\begin{aligned}
\label{Eq:HjkmEst}
\sum _{j=1}^\infty |H(j,k,m)| &\le \left | \frac ts\right |^d,
\quad
\sum _{k=1}^\infty |H(j,k,m)| \le \left | \frac st\right |^d
\\[1ex]
\text{and}\quad
\sum _{m=1}^\infty |H(j,k,m)| &\le |st|^{-d}
\end{aligned}
\end{equation}
while the second inequality follows by
similar arguments, and is left for the reader.

\par

In fact, by arithmetic and geometric mean-value inequality we get
\begin{align}
\sum _{m=1}^\infty |H(j,k,m)|
&\le
\frac 12 \left ( \left ( \sum _{m=1}^\infty H_1(j,k,m)\right )
+
\left ( \sum _{m=1}^\infty H_2(j,k,m)\right )\right ),
\intertext{where}
H_1(j,k,m) &= \int _{\rr {d}}
|(F_{1,j,k}(\cdo ,z),h_{1,m})_{L^2(\rr d)}|^2\, dz
\label{Eq:H1mjkDef}
\intertext{and}
H_2(j,k,m) &= \int _{\rr {d}}
|(h_{2,m},F_{2,j,k}(\cdo ,z))_{L^2(\rr d)}|^2\, dz .
\label{Eq:H2mjkDef}
\end{align}
Since $\{ h_{n,m}\} _{m=1}^\infty$ are orthonormal sequences,
Parseval's formula gives
\begin{align*}
\sum _{m=1}^\infty H_1(j,k,m) &= \sum _{m=1}^\infty  \int _{\rr {d}}
|(F_{1,j,k}(\cdo ,z),h_{1,m})_{L^2(\rr d)}|^2\, dz
\\[1ex]
&\le
\int _{\rr d} \nm {F_{1,j,k}(\cdo ,z)}{L^2}^2\, dz
=
\nm {F_{1,j,k}}{L^2}^2
\\[1ex]
&=
\iint _{\rr {2d}} |f_{1,j}(sx+z/s)g_{2,k}(-tx-z/t)|^2\, dxdz
\\[1ex]
&=
|st|^{-d} \iint _{\rr {2d}} |f_{1,j}(x_1)g_{2,k}(x_2)|^2\, dx_1dx_2
=
|st|^{-d}.
\end{align*}
In the second last equality we have taken $(sx+z/s,-tx-z/t)$
as new variables of integration and are using the fact that $s^{-2}-t^{-2}=1$.
By similar arguments we also obtain
$$
\sum _{m=1}^\infty H_2(j,k,m)\le |st|^{-d}.
$$
A combination
of these estimates gives the third inequality in \eqref{Eq:HjkmEst}.


\medspace

Next we prove the first inequality in \eqref{Eq:HjkmEst}. By
performing the change of variables $x\mapsto sx+z/s$ in the integration in
the scalar product in \eqref{Eq:H1mjkDef} and using that
$s^{-2}-t^{-2}=1$, we obtain
\begin{multline*}
H_1(j,k,m) = \\|s|^{-2d}\int _{\rr {d}}
 |(g_{2,k}(-(t/s)\cdo +tz)\, 
\overline{h_{1,m}((1/s)\cdo - (1/s^2)z) }
,\overline{f_{1,j}} )_{L^2(\rr d)}|^2\, dz.
\end{multline*}
Since $\{\overline {f_{1,j}}\} _{j=1}^\infty$ is an orthonormal sequence, Parseval's
formula gives
\begin{align*}
\sum _{j=1}^\infty &H_1(j,k,m) 
\le
|s|^{-2d}\int _{\rr {d}}
\nm {g_{2,k}(-(t/s)\cdo +tz)\, 
\overline{h_{1,m}((1/s)\cdo - (1/s^2)z) }}{L^2}^2\, dz
\\[1ex]
&=
|s|^{-2d}\iint _{\rr {2d}}
|g_{2,k}(-(t/s)x +tz) 
h_{1,m}((1/s)x - (1/s^2)z)|^2\, dxdz
\\[1ex]
&=
|s|^{-2d}\left | \frac t{s^3}-\frac ts\right |^{-d}\iint _{\rr {2d}}
|g_{2,k}(x_1) 
h_{1,m}(x_2)|^2\, dx_1dx_2
\\[1ex]
&=
|t/s|^d \nm {g_{2,k}}{L^2}^2\nm {h_{1,m}}{L^2}^2=|t/s|^{d}.
\end{align*}
The first inequality in
\eqref{Eq:HjkmEst} now follows by combining these three estimates.

\par

A combination of \eqref{Eq:YoungFuncCond},
\eqref{Eq:ConvExpFormulation} and \eqref{Eq:HjkmEst} gives
\begin{align*}
|(a(s\cdo )*b(t\cdo ),u)|
&\le
(2\pi )^{\frac d2}|st|^{-d}
\left (c_0J_1+c_1J_2+c_2J_3 \right ),
\intertext{where}
J_0
&=
\sum _{j,k,m=1}^\infty
\Phi _1(\lambda _j)\Phi _2(\mu _k)H(j,k,m)\le |st|^{-d},
\\[1ex]
J_1
&=
\sum _{j,k,m=1}^\infty
\Phi _1(\lambda _j)\Phi _0^*(\nu _m)H(j,k,m)\le \left |\frac st \right |^{d}
\intertext{and}
J_2
&=
\sum _{j,k,m=1}^\infty
\Phi _2(\mu _k)\Phi _0^*(\nu _m)H(j,k,m)\le \left |\frac ts \right |^{d}.
\end{align*}
This gives
$$
|(a(s\cdo )*b(t\cdo ),u)|
\le
(2\pi )^{\frac d2}|st|^{-d}
\left (c_0|st|^{-d}+c_1\left | \frac st\right |^d+c_2\left | \frac ts\right |^d \right ),
$$
when $\nm a{s_{\Phi _1}^w}=\nm b{s_{\Phi _2}^w}=\nm u{s_{\Phi _0^*}^w}=1$.
Then inequality \ref{Eq:OrlSchConvL1Eq1} in the case
$a_1,a_2\in \mascS (\rr {2d})$ follows from the latter estimate and duality,
using Proposition \ref{Prop:SchattSymbEquivNorm}.
The asserted extension now follows from \ref{Eq:OrlSchConvL1Eq1}
and Hahn-Banach's theorem. 

\par

If in addition $\Phi _j$ satisfy a local $\Delta _2$-condition for some
$j=1,2$, then 
we may assume that $\Phi _j$ satisfies a global $\Delta _2$-condition.
The uniqueness of the map $(a_1,a_2)\mapsto a_1(t_1\cdo )*a_2(t_2\cdo )$
then follows from the fact that
$\mascS (\rr {2d})$ is dense in $s_{\Phi _j}^w(\rr {2d})$.

\par

Finally, the positivity assertions are straight-forward consequences of
\cite[Corollary 3.5]{Toft3} and density arguments. The details are left for the reader,
and the result follows.
\end{proof}

\par

Using the fact that the $s_\Phi ^w$ classes are norm invariant under symplectic
Fourier transformation, Theorem \ref{Thm:OrlSchConvL2} leads to the following result concerning
dilated multiplications for Orlicz Schatten-von Neumann classes.

%

\par

\begin{thm}\label{Thm:OrlSchMultL2MultLin}
Let $\Phi _j$ be Young functions
such that {\rm{\eqref{Eq:YoungFuncCondGen}}} holds for some
constants $c_j,T>0$, $j=0,\dots ,N$.
Also suppose that $t_1,\dots ,t_N\in \mathbf R\setminus 0$ satisfy
\begin{equation}
\label{Eq:tCondMultilinCase2}
\sum _{j=1}^N  m_jt_j^{2} = 1,
\end{equation}
for some choices of $m_j\in \{ -1,1\}$, $j=1,\dots ,N$.
Then the map
$$
(a_1,\dots ,a_N)\mapsto a_1(t_1\cdo )\cdots a_N(t_N\cdo )
$$
from $\mascS (\rr {2d})\times \cdots \times \mascS (\rr {2d})$ to
$\mascS (\rr {2d})$ extends to a continuous map from
$s_{\Phi _1} ^w(\rr {2d})\times \cdots \times s_{\Phi _N} ^w(\rr {2d})$ to
$s_{\Phi _0} ^w(\rr {2d})$,
and
\begin{equation}
\label{Eq:OrlSchMultL1Eq1Mult}
\begin{aligned}
\nm {a_1(t_1\cdo )&\cdots a_N(t_N\cdo )}{s_{\Phi _0} ^w}
\\
&\le 2\left ( \frac 2\pi \right )^{(N-1)
\frac d2}\left (c_0+\sum _{j=1}^Nc_j|t_j|^{-2d} \right )
\prod _{j=1}^N \left ( |t_j|^{2d}\nm {a_j}{s_{\Phi _j}^w}\right ),
\\[1ex]
a_1&\in s_{\Phi _1} ^w(\rr {2d}),\dots , a_N\in s_{\Phi _N}^w(\rr {2d}).
\end{aligned}
\end{equation}

\par

If in addition at least $N-1$ of $\Phi _1,\dots ,\Phi _N$
satisfy a local $\Delta _2$-condition, then
$a_1(t_1\cdo )*\cdots *a_N(t_N\cdo )$ in
{\rm{\eqref{Eq:OrlSchMultL1Eq1Mult}}}
is uniquely defined.
\end{thm}

\par

\begin{proof}
Let $u_j=\frac 1{t_j}$ and $b_j=\mascF _\sigma a_j$, $j=1,\dots ,N$.
Then
$$
m_1 u_1^{-2}+\cdots +m_N u_N^{-2}=1.
$$
A combination of these identities, the fact that $\mascF _\sigma ^2$ is the
identity operator, Proposition \ref{Prop:SymplFourInvWeylCase},
\eqref{Eq:TransConvMultSymplFourT}, and Theorem
\ref{Thm:OrlSchConvL2MultLin}
gives
\begin{align*}
&\nm {a_1(t_1\cdo )\cdots a_N(t_N\cdo )}{s_{\Phi _0}^w}
=
\pi ^{-(N-1)d}(t_1\cdots t_N)^{-2d}
\nm {b_1(u_1\cdo )* \cdots *b_N(u_N\cdo )}{s_{\Phi _0}^w}
\\[1ex]
&\le
\pi ^{-(N-1)d} 2(2\pi )^{(N-1)\frac d2}(t_1\cdots t_N)^{-2d}
\left (c_0+\sum _{j=1}^Nc_j|u_j|^{2d} \right )
\prod _{j=1}^N \left ( |u_j|^{-2d}\nm {b_j}{s_{\Phi _j}^w}\right )
\\[1ex]
&=
2\left ( \frac 2\pi \right )^{(N-1)
\frac d2}\left (c_0+\sum _{j=1}^Nc_j|t_j|^{-2d} \right )
\prod _{j=1}^N \left ( |t_j|^{2d}\nm {a_j}{s_{\Phi _j}^w}\right ),
\end{align*}
and \eqref{Eq:OrlSchMultL1Eq1Mult} follows. The uniqueness
assertions in the case when at least $N-1$ $\Phi _j$
satisfy a $\Delta _2$-condition, follows by similar arguments as
in the proof of Theorem \ref{Thm:OrlSchMultL2MultLin}, in Appendix
\ref{App:A}.
\end{proof}

\par

\begin{rem}
Let $C>0$ be a constant.
In similar ways as for the convolution properties
in Section \ref{sec3}, it follows that
Theorems \ref{Thm:OrlSchConvL2MultLin}
and \ref{Thm:OrlSchMultL2MultLin}
hold true for $c_j=C$, $j=0,\dots ,N$, after 
the condition \eqref{Eq:YoungFuncCondGen}
is replaced by the condition
\begin{equation}
\label{Eq:YoungCond1MCA}
\prod _{j=1}^N\Phi _j^{-1}(s)
\le
Cs^{N-1}\Phi _0^{-1}(s), \quad s\ge 0.
\end{equation}
This follows from the fact that \eqref{Eq:YoungCond1MCA}
implies \eqref{Eq:YoungFuncCondGen} when
$c_j=C$, $j=0,\dots ,N$,
in view of Proposition \ref{Prop:YoungCondEsts}
in Appendix \ref{App:B}.
\end{rem}

\par

\section{Continuity for Toeplitz operators}
\label{sec5}

\par

In this section we apply our convolution results from
previous sections to deduce Orlicz Schatten-von Neumann
properties of Toeplitz operators, whose symbols should
either belong to Orlicz spaces or dilated Orlicz Schatten-von Neumann
symbols in the Weyl calculus.

\vspace{1mm}\par 

We begin with Toeplitz operators with symbols in Orlicz spaces. 

\par

\begin{prop}
\label{Prop:ToeplOrlContMap}
Let $\phi _1,\phi _2\in L^2(\rr d)$, $\Phi$ be
a Young function which fulfills a local $\Delta _2$-condition,
and let $a\in L^\Phi (\rr {2d})$. Then the map
\begin{equation}
\label{Eq:ToeplMapMultLin}
(a,\phi _1,\phi _2)\mapsto \tp _{\phi _1,\phi _2}(a)
\end{equation}
from $\mascS (\rr {2d})\times \mascS (\rr d)\times \mascS (\rr d)$
to $\maclL (\mascS (\rr d),\mascS '(\rr d))$ is uniquely extendable
to a continuous map from $L^\Phi (\rr {2d})\times L^2(\rr d)\times L^2(\rr d)$
to $\mascI _\Phi (L^2(\rr d))$, and
\begin{equation}
\label{Eq:ToeplCont}
\nm {\tp _{\phi _1,\phi _2}(a)}{\mascI _\Phi}
\le
C\nm a{L^\Phi}\nm {\phi _1}{L^2}\nm {\phi _2}{L^2},
\end{equation}
for some constant $C>0$ which is independent of
$a\in L^\Phi (\rr {2d})$ and $\phi _1,\phi _2\in L^2(\rr d)$.
\end{prop}

\par

%
%
%

\begin{proof}
First suppose that $a\in \mascS '(\rr {2d})$ and
$\phi _1,\phi _2\in \mascS (\rr d)$.
Then $\tp _{\phi _1,\phi _2}(a)$ is uniquely defined as a
continuous operator from $\mascS (\rr d)$ to $\mascS '(\rr d)$, in view of
\eqref{Eq:ToepWeyl}. Since $L^\Phi (\rr {2d})\subseteq \mascS '(\rr {2d})$,
the same holds true when, more restrictive, $a\in L^\Phi (\rr {2d})$. Hence
the result follows from the fact that $\mascS (\rr d)$ is dense in
$L^2(\rr d)$, if we prove that \eqref{Eq:ToeplCont} holds for
$\phi _1,\phi _2\in \mascS (\rr d)$ and $a\in L^\Phi (\rr {2d})$.

\par

Since
$$
\nm {b}{s_1^w}
=
\nm {\phi _1}{L^2}\nm {\phi _2}{L^2},
\quad
b=\check W_{\phi _2,\phi _1},
$$ 
in view of Proposition \ref{Prop:Spectral0Symb},
it follows from Corollary \ref{Cor:OrlSchConvL1} that
\begin{equation*}
\nm {\tp _{\phi _1,\phi _2}(a)}{\mascI _\Phi}
=
(2\pi)^{\frac d2}\nm {\op ^w(a*b)}{\mascI _\Phi}
\asymp
\nm {a*b}{s_\Phi ^w}
\lesssim
\nm {a}{L^\Phi}\nm b{s_1^w}.
\end{equation*}
(See also Remark \ref{Rem:RankOne}.) This gives the result.
%
%
\end{proof}

\par

\begin{prop}
Let $\phi _1,\phi _2\in L^2(\rr d)$, $\Phi$ be
a Young function which fulfills a local $\Delta _2$-condition,
and let $a\in \mascS '(\rr {2d})$ be such that $a(\sqrt 2 \cdo )\in s_\Phi ^w(\rr {2d})$.
Then
$\tp _{\phi _1,\phi _2}(a)\in \mascI _\Phi (L^2(\rr d))$. If in addition
$\phi _1=\phi _2$ and $\op ^w(a(\sqrt 2\cdo ))$ is positive semi-definite, then
$\tp _{\phi _1,\phi _2}(a)$ is positive semi-definite.
\end{prop}

\par

\begin{proof}
Let $a_0=a(\sqrt 2\cdo )$, $t_1=1/\sqrt 2$ and $t_2=1$. Then $t_1^{-2}-t_2^{-2}=1$
and $a_0\in  s_\Phi ^w(\rr {2d})$.
Hence Theorem \ref{Thm:OrlSchConvL2} shows that
$$
a*\check W_{\phi _2,\phi _1}
=
a_0(t_1\cdo )*\check W_{\phi _2,\phi _1}\in s_\Phi ^w(\rr {2d}).
$$
Hence
$$
\tp _{\phi _1,\phi _2}(a)=\op ^w(a*\check W_{\phi _2,\phi _1})
\in \mascI _\Phi (L^2(\rr d)),
$$
giving the first part of the result. 

\par

If in addition $\phi _1=\phi _2$ and $\op ^w(a_0)$ is positive
semi-definite, then the same holds true for
$\op ^w(a_0(t_2\cdo )*\check W_{\phi _1,\phi _1})$
in view of the last part of Theorem \ref{Thm:OrlSchConvL2},
because $\op ^w(\check W_{\phi _1,\phi _1})$ is positive 
semi-definite.
\end{proof}

\par

\appendix

\section{Proof of multi-linear
dilated convolution estimates}
\label{App:A}

\par

In this appendix we give a proof of Theorem
\ref{Thm:OrlSchConvL2MultLin}. It is convenient to use the
mappings
\begin{equation}
\begin{aligned}
\label{Eq:TtDef}
S_t(z,x,y)
&=
\left (tz+\frac {x-y}{2t},tz-\frac {x-y}{2t}\right )\in \rr {2d},
\\[1ex]
T_t(z,x,y) 
&=
\left (tz+\frac {x-y}{2t}+\frac {t(x+y)}{2},
tz-\frac {x-y}{2t}+\frac {t(x+y)}{2}\right )
\in \rr {2d},
\\[1ex]
x,y,z&\in \rr d,\ t\in \mathbf R .
\end{aligned}
\end{equation}
In some situations we will also need to include one more
parameter in $T_t$ above as
\begin{equation}
\label{Eq:TtsDef}
T_{t,s}(z,x,y) 
=
\left (tz+\frac {x-y}{2t}+\frac {s(x+y)}{2},
tz-\frac {x-y}{2t}+\frac {s(x+y)}{2}\right ).
\end{equation}

\par

It follows that \cite[Lemma 3.1]{Toft3} is equivalent to the following lemma.

\par

\begin{lemma}
\label{Lem:DilConvForm1}
Let $a,b\in \mascS (\rr {2d})$, $s,t\in \mathbf R\setminus 0$, $S_t$
and $T_t$ be
given by \eqref{Eq:TtDef}, and let $\maclK$ be as in
\eqref{Eq:OpsDilConv}. Then
\begin{align}
(\maclK &(a(s\cdo ) * b(t\cdo )))(x,y)
\notag
\\[1ex]
&=
\mathfrak c \int \limits _{\rr d}\maclK a(S_s(z,x,y))\, \maclK b(T_t(-z,x,y))\, dz ,
\intertext{where}
\mathfrak c
&=
(2\pi)^d|st|^{-d}.
\end{align}
\end{lemma}

\par

Our first goal is to deduce a multi-dimensional version
of the previous result. In fact we have the following.

\par

\begin{lemma}
\label{Lem:DilConvForm2}
Let $a_j\in \mascS (\rr {2d})$, $t_j\in \mathbf R\setminus 0$, $j=1,\dots ,N$, 
$$
\kappa _N(z) = -\sum _{j=1}^{N-1} z_j,
\qquad
z=(z_1,\dots ,z_{N-1})\in \rr {(N-1)d},\ z_1,\dots ,z_{N-1}\in \rr d,
$$
$S_t$ and $T_t$ be
given by \eqref{Eq:TtDef}, and let $\maclK$ be as in
\eqref{Eq:OpsDilConv}. Then
\begin{align}
\big ( \maclK &(a_1(t_1\cdo ) * \cdots * a_N(t_N\cdo )) \big )(x,y)
\notag
\\[1ex]
&=
\mathfrak c \hspace{-0.4cm} \int \limits _{\rr {(N-1)d}} \left (\prod _{j=1}^{N-1}
\maclK a_j(S_{t_j}(z_j,x,y))
\right )
\maclK a_N\!  \left (
T_{t_N}
\left (\kappa _N(z),x,y \right ) \right )
\, dz ,
\label{Eq:DilConvMultLinCaseIdent}
\intertext{where}
\mathfrak c 
&=
(2\pi)^{(N-1)d}\prod _{j=1}^N |t_j|^{-d}
\quad \text{and}\quad
dz = dz_1\cdots dz_{N-1}
\end{align}
\end{lemma}

\par

\begin{proof}
By Lemma \ref{Lem:DilConvForm1} it follows that the assertion is
true when $N=2$. We shall prove the general case by induction.

\par

Therefore, suppose that the assertion is true for $N=m\ge 2$. We
shall prove that the result holds true for $N=m+1$. First we observe that
\begin{gather*}
a_1(t_1\cdo ) * \cdots * a_m(t_m\cdo ) = |t_1|^{-2(m-1)d}b(t_1\cdo ),
\intertext{where}
b = a_1(s_1\cdo )*\cdots *a_m(s_m\cdo),
\qquad
s_j = \frac {t_j}{t_1},\ j=1,\dots ,m .
\end{gather*}

Since the assertion holds for $N=2$, we obtain
\begin{equation}
\label{Eq:FirstConvReform}
\begin{aligned}
&(2\pi )^{-d}|t_1|^{(2m-1)d}|t_{m+1}|^{d}
(\maclK (a_1(t_1\cdo ) * \cdots * a_{m+1}(t_{m+1}\cdo )))(x,y)
\\[1ex]
&=
(2\pi )^{-d}|t_1t_{m+1}|^{d}
\big (\maclK (b_1(t_1\cdo ) * a_{m+1}(t_{m+1}\cdo )) \big )(x,y)
\\[1ex]
&=
\int \limits _{\rr d}
(\maclK b_1)(S_{t_1}(z_m,x,y))
(\maclK a_{m+1})(T_{t_{m+1}}(-z_m,x,y))\, dz_m.
\end{aligned}
\end{equation}

\par

Next we shall perform suitable reformulations of $\maclK b_1$ in
\eqref{Eq:FirstConvReform}. Let
$$
u,v\in \rr d
\quad \text{and}\quad
z=(z_1,\dots ,z_{m-1})\in \rr {(m-1)d},\ z_1,\dots ,z_{m-1}\in \rr d.
$$
By the inductive assumption we obtain
\begin{equation}
\label{Eq:Kb1Reform}
\begin{aligned}
(\maclK b_1)(u,v)
&=
\big ( \maclK  (a_1(s_1\cdo ) * \cdots *a_m(s_m\cdo ) ) \big )(u,v)
\\[1ex]
&=
(2\pi )^{(m-1)d}|s_2\cdots s_m|^{-d}I(u,v),
\end{aligned}
\end{equation}
where
\begin{align*}
I(u,v)
&=
\int \limits _{\rr {(m-1)d}} \left (\prod _{j=1}^{N-1}
\maclK a_j(S_{s_j}(z_j,u,v))
\right )
\maclK a_m \! \left (
T_{s_m}\!
\left (\kappa _m(z),u,v \right ) \right )
\, dz
\\[1ex]
&=
|t_1|^{(m-1)d}
\int \limits _{\rr {(m-1)d}} \left (\prod _{j=1}^{N-1}
\maclK a_j(S_{t_j}(z_j,t_1u,t_1v))
\right )
H(z,u,v)\, dz.
\intertext{Here}
H(z,u,v)
&=
\maclK a_m \! \left (
T_{t_m,s_m}\!
\left (\kappa _m(z),t_1u,t_1v \right ) \right ),
\end{align*}
$T_{t_m,s_m}$ is given by \eqref{Eq:TtsDef},
and in the last step we have taken $z/t_1$ as new variables
of integration.

\par

If we let
$$
(u,v) =S_{t_1}(z_m,x,y)
=
\left (t_1z_m+\frac {x-y}{2t_1},t_1z_m-\frac {x-y}{2t_1}\right )
$$
as in \eqref{Eq:FirstConvReform}, one obtains
\begin{equation}
\label{Eq:uvAddDiffProp}
u-v = \frac {x-y}{t_1}
\quad \text{and}\quad
\frac {u+v}2=t_1z_m.
\end{equation}
This implies
\begin{equation}
\label{Eq:ExprStj}
S_{t_j}(z_j,t_1u,t_1v)
=
S_{t_j}(z_j,x,y),
\qquad j=1,\dots ,m-1.
\end{equation}
For $j=m$, \eqref{Eq:uvAddDiffProp} gives
\begin{align*}
&T_{t_m,s_m}\!
\left (\kappa _m(z),t_1u,t_1v \right )
\\[1ex]
&=
\left (t_m\kappa _m(z)+\frac {t_1(u-v)}{2t_m}+\frac {s_m(u+v)}{2},
t_m\kappa _m(z)-\frac {t_1(u-v)}{2t_m}+\frac {s_m(u+v)}{2}
\right )
\\[1ex]
&=
\left (t_m(\kappa _m(z)+z_m) +\frac {x-y}{2t_m},
t_m(\kappa _m(z)+z_m) - \frac {x-y}{2t_m}
\right )
\\[1ex]
&=
S_{t_m}(\kappa _{m+1}(z,-z_m),x,y).
\end{align*}

\par

A combination of these identities gives
\begin{align*}
I(u,v)
&=
|t_1|^{(m-1)d}
\int \limits _{\rr {(m-1)d}} \left (\prod _{j=1}^{N-1}
\maclK a_j(S_{t_j}(z_j,x,y))
\right )
H(z,u,v)\, dz
\intertext{where}
H(z,u,v)
&=
\maclK a_m \! \left (
S_{t_m}\!
\left (\kappa _{m+1}(z,-z_m),x,y \right ) \right ) .
\end{align*}
By inserting this into \eqref{Eq:Kb1Reform}, one obtains
\begin{multline*}
(2\pi )^{-(m-1)d}|t_2\cdots t_m|^{d}(\maclK b_1)(S_{t_1}(z_m,x,y))
\\[1ex]
=
\int \limits _{\rr {(m-1)d}} \left (\prod _{j=1}^{N-1}
\maclK a_j(S_{t_j}(z_j,x,y))
\right )
H(z,S_{t_1}(z_m,x,y))\, dz.
\end{multline*}

\par

If we insert the latter expression into \eqref{Eq:FirstConvReform}
and taking $(z_1,\dots ,z_{m-1},-z_m)$ as new variables of integrations
we obtain \eqref{Eq:DilConvMultLinCaseIdent} for $N=m+1$, and the
result follows.
\end{proof}

\par

We shall combine Lemma \ref{Lem:DilConvForm2}
with the following lemma.

\par

\begin{lemma}
\label{Lem:DilConvRankOneElem}
Let $\{ f_{j,k_j}\} _{k_j=1}^\infty \in \ON (L^2(\rr d))$,
$\{ g_{j,k_j}\} _{k_j=1}^\infty \in \ON (L^2(\rr d))$,
$j=1,\dots ,N+1$, and $t_j\in \mathbf R\setminus 0$, be such that
\begin{equation}
\label{Eq:tCondMultilinCaseAgain}
\sum _{j=1}^N \frac {m_j}{t_j^2} = 1,
\end{equation}
for some choices of $m_j\in \{ -1,1\}$, $j=1,\dots ,N$. Also let
\begin{align*}
G(k)
&=
\big ( W_{f_{1,k_1},g_{1,k_1}}(t_1\cdo )*\cdots *W_{f_{N,k_N},g_{N,k_N}}(t_N\cdo ),
W_{f_{N+1,k_{N+1}},g_{N+1,k_{N+1}}}\big )_{L^2}
\\[1ex]
k&=(k_1,\dots ,k_{N+1})\in \zz {N+1}_+.
\end{align*}
Then
\begin{align}
\sum _{k_n=1}^\infty |G(k)|
&\le
(2\pi )^{(N-1)\frac d2}|t_n|^{-2d} \prod _{j=1}^N |t_j|^{-2d},
\quad
n=1,\dots ,N,
\label{Eq:SumHfuncEstFirst}
\intertext{and}
\sum _{k_{N+1}=1}^\infty |G(k)|
&\le
(2\pi )^{(N-1)\frac d2}\prod _{j=1}^N |t_j|^{-2d}.
\label{Eq:SumHfuncEst}
\end{align}
\end{lemma}

\par

\begin{proof}
First we prove \eqref{Eq:SumHfuncEst}. 
It is then convenient to separate out $k_{N+1}$ from $k$, and redefine
$k$ as $k=(k_1,\dots ,k_N)$.
Let
$a_{j,k_j}=W_{f_{j,k_j},g_{j,k_j}}$. Then
$$
\maclK a_{j,k_j}=(2\pi )^{-\frac d2}f_{j,k_j}\otimes g_{j,k_j}.
$$
By Lemma \ref{Lem:DilConvForm2} we get
\begin{align}
G(k,k_{N+1})
&=
\big ( a_{1,k_1}(t_1\cdo )*\cdots *a_{N,k_N}(t_N\cdo ),
a_{N+1,k_{N+1}} \big )_{L^2}
\notag
\\[1ex]
&=
(2\pi )^d
\big ( \maclK (a_{1,k_1}(t_1\cdo )*\cdots *a_{N,k_N}(t_N\cdo )),
\maclK a_{N+1,k_{N+1}}\big )_{L^2}
\label{Eq:GkReform1}
\end{align}
Let
\begin{align}
\fy _{j,k_j} &=
\begin{cases}
\overline{g_{j,k_j}}, & m_j=-1,
\\[1ex]
f_{j,k_j}, & m_j=1,
\end{cases}
\quad
\psi _{j,k_j} =
\begin{cases}
f_{j,k_j}, & m_j=-1,
\\[1ex]
\overline{g_{j,k_j}}, & m_j=1,
\end{cases}
\quad
j=1,\dots ,N,
\\[1ex]
H_1(k,x,y,z)
&=
\left (
\prod _{j=1}^{N-1}\fy _{j,k_j}\left (t_jz_j+m_j\frac {x-y}{2t_j} \right )
\right ) \times
\notag
\\
&\times
\fy _{N,k_N}\left (t_N\kappa _N(z)+m_N
\frac {x-y}{2t_N}+\frac {t_N(x+y)}2 \right ),
\label{Eq:H1Def}
\intertext{and}
H_2(k,x,y,z)
&=
\left (
\prod _{j=1}^{N-1}\psi _{j,k_j}\left (t_jz_j-m_j\frac {x-y}{2t_j} \right )
\right )\times
\notag
\\
&\times
\psi _{N,k_N}\left (t_N\kappa _N(z)-m_N
\frac {x-y}{2t_N}+\frac {t_N(x+y)}2 \right ).
\label{Eq:H2Def}
\end{align}

\par

A combination of Lemma \ref{Lem:DilConvForm2}
and \eqref{Eq:GkReform1} gives
\begin{align}
G(k,k_{N+1})
&=
C \iint  _{\rr {2d}} H_0(k,x,y)\overline{f_{N+1,k_{N+1}}(x)}g_{N+1,k_{N+1}}(y)\, dxdy,
\label{Eq:GkReform}
\intertext{where}
H_0(k,x,y)
&=
\int _{\rr {(N-1)d}} H_1(k,x,y,z)H_2(k,x,y,z)\, dz
\label{Eq:H0Expression1}
\intertext{and}
C
&=
(2\pi )^{(N-1)\frac d2}\prod _{j=1}^N |t_j|^{-d}.
\notag
\end{align}
We shall perform suitable reformulations of $H_0$. For the integral in 
\eqref{Eq:H0Expression1} we perform the substitution
$$
z_j\mapsto z_j+m_j\frac {x+y}{2t_j^2}.
$$
This implies that the different arguments in \eqref{Eq:H1Def}
and \eqref{Eq:H2Def} are transfered as
\begin{align}
t_jz_j+m_j\frac {x-y}{2t_j}
&\mapsto
t_jz_j+m_j\frac {x}{t_j},\qquad j=1,\dots ,N-1,
\label{Eq:TransferNoJPlus}
\\[1ex]
t_jz_j-m_j\frac {x-y}{2t_j}
&\mapsto
t_jz_j+m_j\frac {y}{t_j},\qquad j=1,\dots ,N-1,
\label{Eq:TransferNoJMinus}
\end{align}
\begin{align}
t_N\kappa _N(z)+m_N\frac {x-y}{2t_N}+\frac {t_N(x+y)}2
&\mapsto
t_N\kappa (z)+m_N\frac {x}{t_N},
\label{Eq:TransferNoNPlus}
\\[1ex]
t_N\kappa _N(z)-m_N\frac {x-y}{2t_N}+\frac {t_N(x+y)}2
&\mapsto
t_N\kappa (z)+m_N\frac {y}{t_N}.
\label{Eq:TransferNoNMinus}
\end{align}

\par

In fact, \eqref{Eq:TransferNoJPlus} and
\eqref{Eq:TransferNoJMinus} follow by straight-forward
computations, and are
left for the reader. For the left-hand side of
\eqref{Eq:TransferNoNPlus} we have
\begin{align*}
&t_N\kappa _N(z)+m_N\frac {x-y}{2t_N}+\frac {t_N(x+y)}2
\\[1ex]
&\mapsto
t_N\left (
-\sum _{j=1}^{N-1}
z_j +m_N\frac {x-y}{2t_N^2}+\frac {x+y}2
\left (
1-\sum _{j=1}^{N-1}\frac {m_j}{t_j^2}
\right )
\right )
\\[1ex]
&=
t_N\left (
\kappa _N(z)+\frac x2
\left (
1+ \frac {m_N}{t_N^2} -\sum _{j=1}^{N-1}\frac {m_j}{t_j^2}
\right )
+\frac y2
\left (
1 -\sum _{j=1}^{N}\frac {m_j}{t_j^2}
\right )
\right )
\\[1ex]
&=
t_N\kappa (z)+m_N\frac {x}{t_N}.
\end{align*}
We have used \eqref{Eq:tCondMultilinCaseAgain}
in the last equality. This gives \eqref{Eq:TransferNoNPlus},
and in similar ways one obtains \eqref{Eq:TransferNoNMinus}.
The details are left for the reader.

\par

By \eqref{Eq:GkReform}--\eqref{Eq:TransferNoNMinus}
we get an alternative expression of $H_0$ as
\begin{align}
H_0(k,x,y)
&=
\int _{\rr {(N-1)d}} \widetilde H_1(k,x,z)\widetilde H_2(k,y,z)\, dz,
\tag*{(\ref{Eq:H0Expression1})$'$}
\label{Eq:H0Expression2}
\intertext{where}
\widetilde H_1(k,x,z)
&=
\left (
\prod _{j=1}^{N-1}\fy _{j,k_j}\left (t_jz_j+m_j\frac {x}{t_j} \right )
\right )
\fy _{N,k_N}\left (t_N\kappa _N(z)+m_N
\frac {x}{t_N}\right ),
\label{Eq:H1ModDef}
\intertext{and}
\widetilde H_2(k,y,z)
&=
\left (
\prod _{j=1}^{N-1}\psi _{j,k_j}\left (t_jz_j+m_j\frac {y}{t_j} \right )
\right )
\psi _{N,k_N}\left (t_N\kappa _N(z)+m_N
\frac {y}{t_N}\right ).
\label{Eq:H2ModDef}
\end{align}
By inserting this into \eqref{Eq:GkReform}
we obtain
\begin{align*}
G(k,k_{N+1})
&=
C\int _{\rr {(N-1)d}} F_1(k,z)F_2(k,z)\, dz,
\intertext{where}
F_1(k,z)
&=
(\widetilde H_1(k,\cdo ,z),f_{N+1,k_{N+1}})_{L^2(\rr d)}
\intertext{and}
F_2(k,z)
&=
(\widetilde H_2(k,\cdo ,z),\overline{g_{N+1,k_{N+1}}} )_{L^2(\rr d)}.
\end{align*}

\par

Since $\{ f_{N+1,k_{N+1}} \} _{k_{N+1}=1}^\infty$
and $\{ \overline {g_{N+1,k_{N+1}}} \} _{k_{N+1}=1}^\infty$
are orthonormal sequences, it follows from Bessel's inequality
that
$$
\sum _{k_{N+1}=1}^\infty |F_1(k,z)|^2
=
\sum _{k_{N+1}=1}^\infty |(\widetilde H_1(k,\cdo ,z),f_{N+1,k_{N+1}})_{L^2(\rr d)}|^2
\le
\nm {\widetilde H_1(k,\cdo ,z)}{L^2}^2
$$
and
$$
\sum _{k_{N+1}=1}^\infty |F_2(k,z)|^2
=
\sum _{k_{N+1}=1}^\infty |(\widetilde H_2(k,\cdo ,z),
\overline{g_{N+1,k_{N+1}}})_{L^2(\rr d)}|^2
\le
\nm {\widetilde H_2(k,\cdo ,z)}{L^2}^2.
$$
A combination of these estimates with the identities above,
and Cauchy-Schwarz inequality, we obtain
\begin{align}
C^{-1}\sum _{k_{N+1}=1}^\infty 
|G(k,k_{N+1})|
&\le
\int _{\rr {(N-1)d}} \left (
\sum _{k_{N+1}=1}^\infty |F_1(k,z)| \, |F_2(k,z)| \right )\, dz
\notag
\\[1ex]
&\le
\int _{\rr {(N-1)d}} \left (
\sum _{k_{N+1}=1}^\infty |F_1(k,z)|^2 \right )^{\frac 12}
\left (
\sum _{k_{N+1}=1}^\infty|F_2(k,z)|^2 \right )^{\frac 12}\, dz
\notag
\\[1ex]
&\le
\int _{\rr {(N-1)d}}
\nm {\widetilde H_1(k,\cdo ,z)}{L^2}\nm {\widetilde H_2(k,\cdo ,z)}{L^2}
\, dz
\notag
\\[1ex]
&\le
\nm {\widetilde H_1(k,\cdo )}{L^2}\nm {\widetilde H_2(k,\cdo )}{L^2}.
\label{Eq:GSumEst}
\end{align}

\par

We need to estimate $\nm {\widetilde H_1(k,\cdo )}{L^2}$
and $\nm {\widetilde H_2(k,\cdo )}{L^2}$. By taking the $L^2$-norm
of \eqref{Eq:H1ModDef} and taking
\begin{align*}
u_j
&=
t_jz_j+m_j\frac {x}{t_j} ,\qquad j=1,\dots ,N-1,
\intertext{and}
u_N
&=
\kappa (z)+m_N\frac {x}{t_N}, 
\end{align*}
it follows that
\begin{equation}
\label{Eq:TildeH1Ident}
\nm {\widetilde H_1(k,\cdo )}{L^2}^2 = |D|^{-1} \prod _{j=1}^N \nm {\fy _{j,k_j}}{L^2}^2
= |D|^{-1},
\end{equation}
where $D$ is the determinant given by
$$
D=
\left |
\begin{matrix}
\frac {m_1}{t_1}I_d & t_1I_d & 0 & \dots & 0
\\[1ex]
\frac {m_2}{t_2}I_d & 0 & t_2I_d & \dots & 0
\\[1ex]
\vdots & \vdots & \vdots & \ddots & \vdots
\\[1ex]
\frac {m_{N-1}}{t_{N-1}}I_d & 0 & 0 & \dots & t_{N-1}I_d
\\[1ex]
\frac {m_N}{t_N}I_d & -t_NI_d & -t_NI_d & \dots & -t_NI_d
\end{matrix}
\right |.
$$
Here $I_d$ is the $d\times d$ unit matrix and the zeros
stand for the zero $d\times d$ matrix.
In the last equality in \eqref{Eq:TildeH1Ident} we have
used the fact that all $\fy _{j,k_j}$ are $L^2$ normalized.

\par

By factorizing $t_j$ from row $j$ for $D$, for every $j$ we obtain
$$
D=
(t_1\cdots t_N)^d
\left |
\begin{matrix}
\frac {m_1}{t_1^2}I_d & I_d & 0 & \dots & 0
\\[1ex]
\frac {m_2}{t_2^2}I_d & 0 & I_d & \dots & 0
\\[1ex]
\vdots & \vdots & \vdots & \ddots & \vdots
\\[1ex]
\frac {m_{N-1}}{t_{N-1}^2}I_d & 0 & 0 & \dots & I_d
\\[1ex]
\frac {m_N}{t_N^2}I_d & -I_d & -I_d & \dots & -I_d
\end{matrix}
\right |.
$$
By adding the rows $1,\dots ,N-1$ to row $N$ and
using \eqref{Eq:tCondMultilinCaseAgain} we obtain
$$
D=
(t_1\cdots t_N)^d
\left |
\begin{matrix}
\frac {m_1}{t_1^2}I_d & I_d & 0 & \dots & 0
\\[1ex]
\frac {m_2}{t_2^2}I_d & 0 & I_d & \dots & 0
\\[1ex]
\vdots & \vdots & \vdots & \ddots & \vdots
\\[1ex]
\frac {m_{N-1}}{t_{N-1}^2}I_d & 0 & 0 & \dots & I_d
\\[1ex]
I_d & 0 & 0 & \dots & 0
\end{matrix}
\right |
=
(t_1\cdots t_N)^d.
$$

\par

From these identities it follows that
$\nm {\widetilde H_1(k,\cdo )}{L^2}=|t_1\cdots t_N|^{-\frac d2}$.
In the same way it follows that $\nm {\widetilde H_2(k,\cdo )}{L^2}$
has the same value.
From these identities and \eqref{Eq:GSumEst} we conclude that
$$
\sum _{k_{N+1}=1}^\infty |G(k,k_{N+1})| \le C|t_1\cdots t_N|^{-d},
$$
and the last estimate in \eqref{Eq:SumHfuncEst} follows.

\par

It remains to prove \eqref{Eq:SumHfuncEstFirst}.
Let $a_{j,k_j}$ be the same as above, $b_{k_{N+1}}=a_{N+1,k_{N+1}}$,
$n\in \{ 1,\dots ,N\}$ be fixed and let
\begin{alignat*}{2}
a_{0,j,k_j}
&=
\begin{cases}
a_{j,k_j}, & j\neq n,
\\[1ex]
\widetilde b_{k_{n}},&j=n,
\end{cases}
&\quad
b_{0,k_{N+1}}
&=
\widetilde a_{n,k_{N+1}},
\\[1ex]
s_j
&=
\begin{cases}
\frac {t_j}{t_{n}}, & j\neq n,
\\[1ex]
\frac 1{t_{n}}, & j=n,
\end{cases}
&\qquad \text{and}\qquad
m_{0,j}
&=
\begin{cases}
-m_{n}m_j,& j\neq n,
\\[1ex]
m_{n}, & j=n.
\end{cases}
\end{alignat*}
Here we let $\widetilde a(X)= \overline {a(-X)}$
when $a\in \mascS '(W)$ (see also \eqref{Eq:WeylOpsAdj}).
Then it follows from \eqref{Eq:tCondMultilinCaseAgain} that
\begin{equation}
\tag*{(\ref{Eq:tCondMultilinCaseAgain})$'$}
\label{Eq:CondMultilinCase}
\sum _{j=1}^N \frac {m_{0,j}}{s_j^2} = 1,
\end{equation}
and by a suitable substitution of variables in integrations, one obtains
$$
G(k,k_{N+1}) = |t_{n}|^{-2Nd}(a_{0,1,k_1}(s_1\cdo )
*\cdots *a_{0,N,k_N}(s_N\cdo ),b_{0,k_{N+1}} )_{L^2}.
$$
Since $a_{0,j,k_j}$, $b_{0,k_{N+1}}$ and $s_j$ satisfy similar properties
as $a_{j,k_j}$, $b_{k_{N+1}}$ and $t_j$, it follows from the first
part of the proof that
\begin{align*}
\sum _{k_{n}=1}^\infty |G(k,k_{N+1})|
&=
|t_{n}|^{-2Nd}
\sum _{k_{N+1}=1}^\infty
(a_{0,1,k_1}(s_1\cdo )
*\cdots *a_{0,N,k_N}(s_N\cdo ),b_{0,k_{N+1}} )_{L^2}
\\[1ex]
&=
|t_{n}|^{-2Nd}
(2\pi )^{(N-1)\frac d2}\prod _{j=1}^N |s_j|^{-2d}
\\[1ex]
&=
(2\pi )^{(N-1)\frac d2}|t_n|^{2d}\prod _{j=1}^N |t_j|^{-2d}.
\end{align*}
Hence \eqref{Eq:SumHfuncEstFirst} holds,
and the result follows.
\end{proof}

\par

%

\begin{proof}[Proof of Theorem \ref{Thm:OrlSchConvL2MultLin}]
With the same notations as in Lemma
\ref{Lem:DilConvRankOneElem} and its proof,
we first assume that $a_1,\dots ,a_N,b$ belong to
$\mascS (\rr {2d})$, and satisfy
\begin{equation}
\label{Eq:NormalizSchattElem}
\nm {a_j}{s_{\Phi _j}^w}\le 1
\quad \text{and}\quad
\nm {b}{s_{\Phi _0^*}^w}\le 1.
\end{equation}
Then
\begin{align*}
a_j
&=
\sum _{k_j=1}^\infty \lambda _{j,k_j}W_{f_{j,k_j},g_{j,k_j}},
\quad
j=1,\dots ,N,
\intertext{and}
b
&=
\sum _{k_{N+1}=1}^\infty \overline{\lambda _{N+1,k_{N+1}}}
W_{f_{N+1,k_{N+1}},g_{N+1,k_{N+1}}},
\end{align*}
for some sequences
$$
\{ f_{j,k_j}\} _{k_j=1}^\infty , \{ g_{j,k_j}\} _{k_j=1}^\infty \in \ON(L^2(\rr d))
\quad \text{and}\quad
\{ \lambda _{j,k_j}\} _{k_j=1}^\infty \in \mathbf R_+\cup \{ 0\} ,
\quad j=1,\dots ,N+1.
$$
Then $|\lambda _{j,k_j}|\lesssim |k_j|^{-N}$ for every $N\ge 0$, in view of
\cite[Theorem 6.1]{CheSigTof}. Furthermore, \eqref{Eq:NormalizSchattElem} implies
\begin{equation}
\label{Eq:YoungSumsNormaliz}
\sum _{k_j=1}^\infty \Phi _j(\lambda _{j,k_j})\le 1
\quad \text{and}\quad
\sum _{k_{N+1}=1}^\infty \Phi _0^*(\lambda _{N+1,k_{N+1}})\le 1.
\end{equation}

\par

Let
$$
Q\equiv (a(t_1\cdo )*\cdots *a(t_N\cdo ),b)_{L^2}
=
\sum _{k\in \zz {N+1}_+}
\left ( 
\prod _{j=1}^{N+1} \lambda _{j,k_j}
\right )
G(k).
$$
Hence, if $\Omega _j=\{1,\dots ,N\} \setminus \{ j \}$,
$j=1,\dots ,N$, then
\begin{align*}
|Q|
&\le
\sum _{k\in \zz {N+1}_+}
\left ( 
\prod _{j=1}^{N+1} |\lambda _{j,k_j}|
\right )
|G(k)|
\le
\sum _{j=1}^{N+1}c_jR_j,
\intertext{where $c_{N+1}=c_0$,}
R_j
&=
\sum _{k\in \zz {N+1}_+}
\left (
\prod _{m\in \Omega _j}\Phi _m(|\lambda _{m,k_m}|)
\right )
\Phi _0^*(\lambda _{N+1,k_{N+1}})|G(k)|,
\quad j=1,\dots ,N,
\intertext{and}
R_{N+1}
&=
\sum _{k\in \zz {N+1}_+}
\left (
\prod _{j=1}^N\Phi _j(\lambda _{j,k_j})
\right )|G(k)| .
\end{align*}

\par

By Lemma \ref{Lem:DilConvRankOneElem} and
\eqref{Eq:YoungSumsNormaliz}, one obtains
\begin{align}
R_j
&=
\sum
\left (
\prod _{m\in \Omega _j}\Phi _m(|\lambda _{m,k_m}|)
\right )
\Phi _0^*(\lambda _{N+1,k_{N+1}})
\left (
\sum _{k_j=1}^\infty|G(k)|
\right )
\notag
\\[1ex]
&\le
(2\pi )^{(N-1)\frac d2}|t_j|^{2d}
\left (
\prod _{n=1}^N |t_n|^{-2d}
\right ) \times
\notag
\\
&\times
\prod _{m\in \Omega _j}
\left (
\sum _{k_m=1}^\infty \Phi _m(\lambda _{m,k_m})
\right )
\left (
\sum _{k_{N+1}=1}^\infty \Phi _0^*(\lambda _{N+1,k_{N+1}})
\right )
\notag
\\[1ex]
&\le
(2\pi )^{(N-1)\frac d2}|t_j|^{2d}
\left (
\prod _{n=1}^N |t_n|^{-2d}
\right ),
\label{Eq:RjComp}
\end{align}
when $j=1,\dots ,N$. Here the first sum
in \eqref{Eq:RjComp} is taken over all
$$
k_1,\dots ,k_{j-1},k_{j+1},\dots ,k_{N+1}\in \mathbf Z_+.
$$
By similar arguments we get
\begin{align*}
R_{N+1}
&=
\sum _{k\in \zz {N}_+}
\left (
\prod _{j=1}^N\Phi _j(|\lambda _{j,k_j}|)
\right )
\left (
\sum _{k_{N+1}=1}^\infty|G(k)|
\right )
\\[1ex]
&\le
(2\pi )^{(N-1)\frac d2}
\left (
\prod _{n=1}^N |t_n|^{-2d}
\right ) 
\prod _{j=1}^N
\left (
\sum _{k_j=1}^\infty \Phi _j(\lambda _{j,k_j})
\right )
\\[1ex]
&\le
(2\pi )^{(N-1)\frac d2}
\left (
\prod _{n=1}^N |t_n|^{-2d}
\right ).
\end{align*}

\par

By combining these estimates one obtains
\begin{multline}
\label{Eq:BasicFormEst}
|(a(t_1\cdo )*\cdots *a(t_N\cdo ),b)_{L^2}|
\\[1ex]
\le (2\pi )^{(N-1)\frac d2}\left (c_0+\sum _{j=1}^Nc_j|t_j|^{2d} \right )
\left (
\prod _{j=1}^N \left ( |t_j|^{-2d}\nm {a_j}{s_{\Phi _j}^w}\right )
\right )\nm b{s_{\Phi _0^*}^w},
\end{multline}
when $a_1,\dots ,a_N,b\in \mascS (\rr {2d})$ satisfy
\eqref{Eq:NormalizSchattElem}.
By homogeneity it follows that \eqref{Eq:BasicFormEst} holds
for any $a_1,\dots ,a_N,b\in \mascS (\rr {2d})$. Repeated
applications of Hahn-Banach's theorem now shows that the definition of
$(a(t_1\cdo )*\cdots *a(t_N\cdo ),b)_{L^2}$ extends to any
$a_j\in s_{\Phi _j}^w(\rr {2d})$ and $b\in s_{\Phi _0^*}^w$,
and that \eqref{Eq:BasicFormEst} still holds for such $a_j$ and $b$,
$j=1,\dots ,N$.

\par

By taking the supremum on the left-hand side of
\eqref{Eq:BasicFormEst} with respect to all
$b\in s_{\Phi _0^*}^w$ with $\nm {b}{s_{\Phi _0^*}^w}\le 1$,
it follows that $a_1(t_1\cdo )*\cdots *a_N(t_N\cdo )$
is well-defined as an element in $s_{\Phi _0}^w$,
and that \eqref{Eq:OrlSchConvL1Eq1} holds.
This shows the asserted extension.

\par

Next suppose that at least $N-1$ of $\Phi _1,\dots ,\Phi _N$
satisfy a local $\Delta _2$-condition. By symmetry we may assume
that $\Phi _1,\dots ,\Phi _{N-1}$ satisfy a local $\Delta _2$-condition.
Let $a_j\in \mascS (\rr {2d})$ and $a_N\in s_{\Phi _N}^w(\rr {2d})$,
$j=1,\dots ,N-1$.
Then $a_1(t_1\cdo )*\cdots *a_N(t_N\cdo )$ is uniquely defined
as an element in $\mascS '(\rr {2d})$, and thereby uniquely
defined as an element in $s_{\Phi _0}^w(\rr {2d})$
in view of \eqref{Eq:OrlSchConvL1Eq1}. The uniqueness of
$a_1(t_1\cdo )*\cdots *a_N(t_N\cdo )$ for arbitrary
$a_j\in s_{\Phi _j}^w (\rr {2d})$, $j=1,\dots ,N$, now follows from the
fact that $\mascS (\rr {2d})$ is dense in $s_{\Phi _j}^w (\rr {2d})$
when $j=1,\dots ,N-1$.

\par

Finally, the positivity assertions in the last part follow from
\cite[Corollary 3.5]{Toft3}, giving the result.
\end{proof}

\par

\section{Proofs of Young condition
relations}\label{App:B}

\par

In this appendix we show that the conditions
\eqref{Eq:YoungCond1} and \eqref{Eq:YoungFuncCondGen}
on Young functions, are more relaxed than
\eqref{Eq:YoungCond2} and \eqref{Eq:YoungCond1MCA},
respectively.


\par


In fact we have the following.

\par

\begin{prop}\label{Prop:YoungCondEsts}
Suppose $\Phi _j$, $j=0,1,\dots ,N$, are Young functions 
which satisfy
\begin{equation}
\tag*{(\ref{Eq:YoungCond1})$'$}
\label{Eq:YoungCond1MC}
\prod _{j=1}^N\Phi _j^{-1}(s)
\le
Cs^{N-1}\Phi _0^{-1}(s), \quad s\ge 0,
\end{equation}
and let
$$
\Theta _N = \sets {(k_1,\dots ,k_{N-1})\in \zz {N-1}}
{1\le k_1<k_2<\cdots <k_{N-1}\le N}.
$$
Then
\begin{equation}
\tag*{(\ref{Eq:YoungCond2})$'$}
\label{Eq:YoungCond2MC}
\begin{aligned}
t_0t_1\cdots t_N
&\le
C\left (
\Phi _0^*(t_0)\sum _{k\in \Theta _N}
\prod _{j=1}^{N-1}\Phi _{k_j}(t_{k_j})
+ \prod _{j=1}^{N}\Phi _{j}(t_{j})
\right  ),
\\[1ex]
t_0,t_1,\dots ,t_N &\ge 0.
\end{aligned}
\end{equation}
\end{prop}

\par

For the proof of Proposition \ref{Prop:YoungCondEsts}
we have the following lemma, where the second part
refines \ref{Eq:YoungCond1MC}.

\par

\begin{lemma}
\label{Lemma:YoungCondRefineMC}
Let $\Phi _j$, $j=0,1,\dots N$, be Young functions.
\begin{enumerate}
\item If $s\ge 0$ and $t>0$, then
$$
\Phi _j^{-1}(s) \le
\frac {\Phi _j^*(t)+s}t.
$$

\vrum

\item If {\rm{\ref{Eq:YoungCond1MC}}} holds for some constant
$C>0$, $j_0\in \{ 1,\dots ,N\}$ and $I =\{ 1,\dots ,N\}\setminus \{j_0\}$,
then
\begin{alignat}{3}
\prod _{j=1}^N\Phi _j^{-1}(s_j)
&\le
C\left (\prod _{j\in I}s_j\right )\Phi _0^{-1}(s_{j_0}),&
\quad &\text{when} &\quad
0&\le s_{j_0}\le s_j,\ j\in I.
\tag*{(\ref{Eq:YoungCond1})$''$}
\label{Eq:YoungCond1Ref1MC}
\end{alignat}
\end{enumerate}
\end{lemma}

\par

\begin{proof}
By the definition of the Young conjugate, one has
$$
\Phi _j^*(t) \ge t \Phi _j^{-1}(s)-s,
$$
which gives (1).

\par

Since $\Phi _j^{-1}$ is concave we have
$$
\frac {\Phi _j^{-1}(s_j)}{s_j}
\le
\frac {\Phi _j^{-1}(s_{j_0})}{s_{j_0}}.
$$
A combination of the latter inequality and \ref{Eq:YoungCond1MC}
gives
\begin{align*}
\prod _{j=1}^N\Phi _j^{-1}(s_j)
&\le
\left (\prod _{j=1}^N\Phi _j^{-1}(s_{j_0})\right )s_{j_0}^{-(N-1)}
\left (
\prod _{j\in I}s_j
\right )
\\[1ex]
&\le
Cs_{j_0}^{N-1}\Phi _0^{-1}(s_{j_0})s_{j_0}^{-(N-1)}
\left (
\prod _{j\in I}s_j
\right )
=
C\left (
\prod _{j\in I}s_j
\right )
\Phi _0^{-1}(s_{j_0}),
\end{align*}
which gives \ref{Eq:YoungCond1Ref1MC}, and thereby the result.
\end{proof}

\par

%

\begin{proof}[Proof of Proposition \ref{Prop:YoungCondEsts}]
We use the same notation as in the proof of
Lemma \ref{Lemma:YoungCondRefineMC}.
Suppose that \ref{Eq:YoungCond1MC} holds. We may assume that
$t_0,t_1,\dots ,t_N>0$ when proving \ref{Eq:YoungCond2MC}.
Let $j_0\in \{1,\dots ,N\}$ be chosen such that $\Phi _{j_0}(t_{j_0})\le \Phi _j(t_j)$
for every $j$. Then we obtain by letting
$s_j=\Phi _j(t_j)$ in Lemma \ref{Lemma:YoungCondRefineMC}
that
\begin{align*}
t_1\cdots t_N &= \Phi _1^{-1}(s_1)\cdots \Phi _N^{-1}(s_N)
\\[1ex]
&\le C \left (\prod _{j\in I}s_j\right )\Phi _0^{-1}(s_{j_0})
= C \left (\prod _{j\in I}\Phi _j(t_j)
\right )\Phi _0^{-1}(s_{j_0})
\\[1ex]
&\le
C \left (\prod _{j\in I}\Phi _j(t_j)
\right ) \frac {\Phi _0^*(t_0)+\Phi _{j_0}(t_{j_0})}{t_0}.
\end{align*}
By multiplying with $t_0$, we obtain
\begin{align*}
t_0t_1\cdots t_N
&\le
C\left (
\Phi _0^*(t_0) \prod _{j\in I}\Phi _j(t_j)
+ \prod _{j=1}^N\Phi _j(t_j)
\right )
\\[1ex]
&\le
C\left (
\Phi _0^*(t_0)\sum _{k\in \Theta _N} \prod _{j=1}^{N-1}\Phi _{k_j}(t_{k_j})
+ \prod _{j=1}^{N}\Phi _{j}(t_{j})
\right  ),
\end{align*}
and \ref{Eq:YoungCond2MC} follows.
This gives the result.
\end{proof}

\par

Next we give a proof of the fact that the condition
\eqref{Eq:HolderYoungFuncCond2}
implies \eqref{Eq:HolderYoungFuncCond}. More generally
we have the following.

\par

\begin{prop}
\label{Prop:HolderYoungFuncCondEsts}
Let $\Phi _j$, $j\in \{ 0,\dots ,N\}$, be increasing functions
on $[0,\infty)$ such that
$$
\Phi _j(0) =0,
\quad
\lim _{t\to \infty} \Phi _j(t) =\infty ,
\qquad
j=0,\dots ,N
$$
and
$$
\prod _{j=1}^{N}\Phi _j^{-1}(s) \le \Phi _0^{-1}(s),
\qquad s\ge 0.
$$
Then
$$
\Phi _0(t_1\cdots t_N) \le \max _{1\le j\le N}\big ( \Phi _j(t_j)\big )
\le \sum _{j=1}^N\Phi _j(t_j).
$$
\end{prop}

\par

\begin{proof}
Let $u_j=\Phi _j(t_j)$, $\fy _j=\Phi _j^{-1}$, $j=1,\dots ,N$ and let $j_0\in \{ 1,\dots ,N\}$
be chosen such that $u_{j_0}=\max (u_j)$. Then
\begin{align*}
\Phi _0(t_1\cdots t_N)
&=
\Phi _0 \big ( \fy _1(u_1)\cdots \fy _N(u_N) \big )
\\[1ex]
&\le
\Phi _0 \big ( \fy _1(u_{j_0})\cdots \fy _N(u_{j_0}) \big )
\\[1ex]
&\le
\Phi _0 \big ( \Phi  _0^{-1}(u_{j_0}) \big ) = u_{j_0} =\max _{1\le j\le N}\big ( \Phi _j(t_j) \big ).
\qedhere
\end{align*}
\end{proof}

\par

\end{document}